\documentclass[11pt]{amsart}
\usepackage{geometry}
\usepackage[noadjust]{cite}

\usepackage{graphicx, xcolor}
\usepackage{amssymb, mathrsfs}
\usepackage{epstopdf}
\usepackage{stmaryrd}
\usepackage{array}
\usepackage{multirow}
\usepackage{hyperref}
\usepackage[new]{old-arrows}
\usepackage{nicefrac}
\usepackage{ytableau}
\usepackage{multirow, lipsum, stmaryrd}

\numberwithin{equation}{section}

\theoremstyle{plain}
\newtheorem{thm}{Theorem}[section]
\newtheorem{lem}[thm]{Lemma}
\newtheorem{prop}[thm]{Proposition}
\newtheorem{cor}[thm]{Corollary}

\newtheorem{thmx}{Theorem}

\theoremstyle{definition}
\newtheorem{defn}[thm]{Definition}

\newtheorem{ex}[thm]{Example}

\theoremstyle{remark}
\newtheorem{rmk}[thm]{Remark}

\usepackage[titletoc]{appendix}
\usepackage[english]{babel}
\usepackage{amsmath, amssymb, amsfonts, amsthm}
\usepackage{graphicx}
\usepackage[hang,small,bf]{caption}
\usepackage{mathtools}
\usepackage[all,cmtip]{xy}
\usepackage[]{units}
\usepackage{amssymb,graphicx}
\usepackage{pictexwd,dcpic}
\usepackage{tikz}
\usepackage{enumerate}
\usepackage{tikz-cd}
\usetikzlibrary{matrix,arrows}
\usepackage{cancel}
\usetikzlibrary{calc}
\usepackage{stmaryrd}
\usepackage{enumitem}

\usepackage{tikz-cd}

\usepackage{comment} 

\DeclareMathOperator{\Hom}{Hom}									
\DeclareMathOperator{\HHom}{\mathcal{H}om}					
\DeclareMathOperator{\End}{End}									
\DeclareMathOperator{\EEnd}{\mathcal{E}nd}	
\DeclareMathOperator{\Aut}{Aut}										
\DeclareMathOperator{\Coh}{\mathsf{Coh}}									

\DeclareMathOperator{\Pic}{Pic}									
\newcommand{\OO}[1]{\mathcal{O}_{#1}}							
\newcommand{\Ztwo}{\nicefrac{\mathbb{Z}}{2\mathbb{Z}}}					
\DeclareMathOperator{\Ext}{Ext}										
\DeclareMathOperator{\EExt}{\mathcal{E}xt}						
\DeclareMathOperator{\NS}{NS}       				
\DeclareMathOperator{\Amp}{Amp}  						
\DeclareMathOperator{\rk}{rk} 						
\DeclareMathOperator{\im}{Im} 										
 	
\DeclareMathOperator{\ch}{ch}						
					
\DeclareMathOperator{\Sym}{Sym}

\DeclareMathOperator{\ide}{id}						
\DeclareMathOperator{\GL}{GL}									
\DeclareMathOperator{\gl}{\mathfrak{gl}}							
\DeclareMathOperator{\SL}{SL}										
\DeclareMathOperator{\slie}{\mathfrak{sl}}						
\DeclareMathOperator{\SP}{Sp}									
\DeclareMathOperator{\splie}{\mathfrak{sp}}						
\DeclareMathOperator{\Grass}{G}							
	
\DeclareMathOperator{\Div}{Div}	

\usepackage{hyperref}

\definecolor{imperial}{rgb}{0.015625, 0.25, 0.4375}
\definecolor{denim}{rgb}{0.08, 0.38, 0.74}

\hypersetup{
    colorlinks=true,
    linkcolor=imperial,
    citecolor=imperial,
    filecolor=magenta,      
    urlcolor=denim
    }

\title{Holomorphic symplectic manifolds from semistable Higgs bundles}
\author{Roland Abuaf and Riccardo Carini}

\calclayout
\begin{document}

\begin{abstract}
Let $\mathcal{M}_{C}(2, 0)$ be the moduli space of semistable rank two and degree zero Higgs bundles on a smooth complex hyperelliptic curve $C$ of genus three. We prove that the quotient of $\mathcal{M}_{C}(2, 0)$ by a twisted version of the hyperelliptic involution is an 18-dimensional holomorphic symplectic variety admitting a crepant resolution, whose local model was studied by Kaledin and Lehn to describe O'Grady's singularities. Similarly, by considering the moduli space of Higgs bundles with trivial determinant $\mathcal{M}_C(2, \OO{C})\subseteq \mathcal{M}_C(2, 0)$, we show that the quotient of $\mathcal{M}_C(2, \OO{C})$ by the hyperelliptic involution is a 12-dimensional holomorphic symplectic variety admitting a crepant resolution. 
\end{abstract}
\maketitle

\addtocontents{toc}{\protect\setcounter{tocdepth}{-1}}
\section*{Introduction}

Let $C$ be a smooth complex projective curve of genus $g\geq2$ and let $\mathcal{M}_{C}(2, 0)$ be the moduli space of $\GL(2,\mathbb{C})$-Higgs bundles on $C$. Its closed points are given by equivalence classes of pairs $(E,\phi)$, where $E$ is a rank two bundle on $C$ of degree zero and $\phi\in\text{Hom}(E,E\otimes \omega_{C})\cong\text{Ext}^{1}(E,E)^{\ast}$, subject to a suitable stability condition. 
Hence, it is a partial compactification of the total space of the cotangent bundle to the moduli space $\mathcal{U}_{C}^{s}(2, 0)$ of stable rank two and degree zero bundles on $C$. In particular, it is naturally endowed with a symplectic form on the smooth locus and is an instance of a holomorphic symplectic variety in the sense of \cite{Beauville-symplectic-singularities}.

Assume $C$ is hyperelliptic, i.e. is endowed with a two-to-one morphism to $\mathbb{P}^{1}$ and denote by $\sigma \in\text{Aut}(C)$ the covering involution. When $g=3$, the action of $\sigma$ can be lifted to the moduli space $\mathcal{M}_{C}(2, 0)$  so that the quotient admits a symplectic resolution:

\begin{thmx}\label{main_thm_GL}
Let $C$ be a smooth hyperelliptic curve of genus three and let $\sigma \in\operatorname{Aut}(C)$ be the hyperelliptic involution. Then the quotient of $\mathcal{M}_{C}(2, 0)$ by the $\Ztwo$-action generated by the symplectic involution
\begin{equation}\label{intro_action}
(E,\phi)\mapsto(\sigma^{\ast}E^\ast , -\sigma^{\ast}\phi^t)
\end{equation}
is an 18-dimensional holomorphic symplectic variety which admits a symplectic resolution obtained by blowing-up the reduced singular locus. 
\end{thmx}

By restricting to Higgs bundles with $\operatorname{tr}\phi=0$ and $\det E\cong \OO{C}$, one gets the holomorphic symplectic subvariety $\mathcal{M}_C(2, \OO{C})\subseteq \mathcal{M}_C(2,0)$, which parametrises $\SL(2,\mathbb{C})$-Higgs bundles and is stable under the $\Ztwo$-action described above.
In this setting, we get the following:

\begin{thmx}\label{main_thm_SL}
Let $C$ be a smooth hyperelliptic curve of genus three and let $\sigma \in\operatorname{Aut}(C)$ be the hyperelliptic involution.
The quotient of $\mathcal{M}_{C}(2, \OO{C})$ by
the $\Ztwo$-action generated by the symplectic involution
\begin{equation}
(E,\phi)\mapsto(\sigma^{\ast}E , \sigma^{\ast}\phi)
\end{equation}
is a 12-dimensional holomorphic symplectic variety which admits a symplectic resolution obtained by blowing-up the reduced singular locus. 
\end{thmx}  

Since the moduli space $\mathcal{M}_C(2, \OO{C})$ is the fibre over $(\OO{C}, 0)$ of the isotrivial fibration
\begin{equation}\label{isotrivial-fibration-intro}
\begin{split}
\mathcal{M}_C(2, 0) & \to T^\ast \Pic^0(C), \\
(E,\phi) & \mapsto (\det(E), \operatorname{tr} \phi),
\end{split}
\end{equation}  
which is invariant with respect to \eqref{intro_action}, the morphism \eqref{isotrivial-fibration-intro} descends to the quotient. Hence $\mathcal{M}_C(2, \OO{C})/\Ztwo$ is also the fibre of an isotrivial fibration defined on $\mathcal{M}_C(2, 0)/\Ztwo$, which is trivialised after a single \'etale cover. In particular the two holomorphic symplectic varieties arising from Theorems \ref{main_thm_GL} and \ref{main_thm_SL} are stably isosingular in the sense of \cite[Def. 2.6]{Mirko}.

\subsection*{Symplectic varieties and crepant resolutions} Recall a holomorphic symplectic variety is a normal variety $X$ with rational Gorenstein singularities whose regular locus carries a symplectic form $\omega \in H^0(X_{\text{reg}}, \Omega^2_{X_{\text{reg}}})$  \cite{Beauville-symplectic-singularities, Namikawa2001}.

A result of Namikawa \cite[Thm. 4]{Namikawa2001} shows that for every resolution of singularities $f: \tilde{X}\to X$, the symplectic form $f^\ast \omega$ extends to a regular form on $\tilde{X}$. We say a resolution $f: \tilde{X}\to X$ is symplectic if such an extension is symplectic. By \cite[Prop. 3.2]{Kaledin-crepant}, a resolution $f: \tilde{X}\to X$ of a symplectic variety $X$ is symplectic if and only if it is crepant, i.e. the canonical morphism $f^\ast \omega_X\to \omega_{\tilde{X}}$ is an isomorphism. 

\subsection*{Crepant resolutions of moduli spaces of Higgs bundles} We are interested in Theorems \ref{main_thm_GL} and \ref{main_thm_SL} because singular moduli spaces of semistable $\GL(2,\mathbb{C})$ or $\SL(2,\mathbb{C})$-Higgs bundles are instances of holomorphic symplectic varieties which cannot give rise to new holomorphic symplectic manifolds via crepant resolutions due to the nature of their singularities (see the no-go results in \cite{Kiem, KL, Kaledin-Lehn-Sorger} and \cite{Tirelli}). An exception is given by the genus $g=2$ case, which is analogous to O'Grady's exceptional examples  \cite{O'Grady10, O'Grady6} in the realm of K3 surfaces (see Theorem \ref{singularities-GL-Higgs}). 

Nevertheless, this leaves open the possibility of exploiting finite symplectic group actions on such moduli spaces, and looking for symplectic resolutions after passing to the fixed loci or quotients. Indeed, the explicit local model around the worst singularity of $\mathcal{M}_{C}(2, 0)$ given in \cite{KL} -- where it was used to describe O'Grady's exceptional examples -- has a natural symplectic resolution after passing to a quotient by a symplectic involution: this is what we use in Theorems \ref{main_thm_GL} and  \ref{main_thm_SL}, by finding a suitable global extension of this local involution. 

Note that this strategy has been already highlighted in \cite{Markushevich}, \cite{Sacca-Prym} and the more recent \cite{Brakkee-Camere-etc}, where fixed loci of symplectic involutions on compact moduli spaces of pure one-dimensional sheaves on K3 surfaces are studied, in order to construct families of Prym varieties. However, in the compact setting, it seems almost impossible to avoid the existence of stable sheaves which become unstable after being pulled back by the relevant symplectic involution  \cite[\textsection 3.5]{Sacca-Prym}. In the Higgs bundles case, such difficulties do not appear (see Corollary \ref{dualising-bundle}) and we are able to fully carry out the aforementioned strategy.

\subsection*{A different formulation} Since we can realise $\mathcal{M}_{C}(2, \OO{C})$ inside $\mathcal{M}_{C}(2,0)$
as a connected component of the fixed locus of the symplectic
involution $\tau$ defined by 
\[
(E,\phi)\mapsto(E^{\ast},-\phi^{t}),
\]
we can rewrite Theorems \ref{main_thm_GL} and \ref{main_thm_SL} in the following form:

\begin{thmx}\label{reformulation}
Let $C$ be a smooth hyperelliptic curve of genus three. Then there exist two commuting symplectic involutions $\tau$ and $\sigma$
on the moduli space $\mathcal{M}_{C}(2,0)$ such that:
\begin{enumerate}[label={\upshape(\roman*)}]
\item The quotient $\mathcal{M}_C(2,0)/ \langle \sigma\circ \tau \rangle$ is an 18-dimensional holomorphic symplectic variety admitting a symplectic resolution obtained by blowing-up the reduced singular locus.
\item There is a connected component $\mathcal{M}_C(2, \OO{C})$ of the fixed locus of $\tau$ which is a holomorphic symplectic subvariety of $\mathcal{M}_C(2,0)$ and is acted on by $\sigma$.
\item The quotient of  $\mathcal{M}_C(2,\OO{C})$ by the action of $\sigma$ is a 12-dimensional holomorphic symplectic variety admitting a symplectic resolution obtained by blowing-up the reduced singular locus. 
\end{enumerate}
\end{thmx}

\subsection*{Beauville--Mukai systems} This reformulation has the advantage of suggesting a way of adapting our construction to moduli spaces of pure one-dimensional sheaves on compact symplectic surfaces. Indeed, the $\GL(2,\mathbb{C})$-Hitchin system naturally deforms to a Beauville--Mukai system (\cite{Donagi}), while there is no a priori obvious candidate for an analogue of the $\SL(2,\mathbb{C})$-Hitchin system. See also \cite[Appendix]{Felisetti-Mauri}, \cite{Franco} and \cite{franco2022ogrady} for an account of the problem. 

It then seems natural to ask whether there are compact holomorphic symplectic manifolds arising from desingularising quotients of a suitable Beauville--Mukai systems which degenerate to the non-compact examples provided by Theorems \ref{main_thm_GL} and \ref{main_thm_SL}. This is the main subject of ongoing research.

\subsection*{Outline} 
In Section \ref{section_symplectic_varieties} we gather standard definitions and results about holomorphic symplectic varieties and crepant resolutions.

In Section \ref{section_local_model} we describe the algebraic local model we will use later on, which is given by the closure of a nilpotent orbit in the symplectic Lie algebra $\mathfrak{sp}(2n, \mathbb{C})$. These varieties have been classically widely studied (\cite{Kraft-Procesi}, \cite{Panyushev}, \cite{Fu}), but they were first exploited in \cite{KL} to describe singularities of moduli spaces of sheaves on symplectic surfaces. 

In Section \ref{section_equivariant_hyperelliptic} we study equivariant bundles on hyperelliptic curves, with respect to the $\Ztwo$-action given by the hyperelliptic involution. Along the way, we also describe the induced action on the moduli spaces of rank two bundles and the corresponding quotients. 

In Section \ref{section_Higgs} we introduce moduli spaces of Higgs bundles and the involutions mentioned in Theorems \ref{main_thm_GL}, \ref{main_thm_SL} and \ref{reformulation}, showing that the $\SL(2,\mathbb{C})$-locus can be realised as the fixed locus of one of them. 

Finally in Section \ref{section_final_proofs}, we relate an analytic neighbourhood of the worst singularity of the moduli spaces to the local model described in Section \ref{section_local_model}, so that it admits a crepant resolution after taking the quotient by a local involution.  We then prove that such an involution can be extended to the whole moduli space by means of the hyperelliptic involution. 

\subsection*{Notation and conventions} 
Throughout the paper we work over the field $\mathbb{C}$ of complex numbers. A \emph{scheme} will always be a scheme of finite type over $\mathbb{C}$, and by \emph{algebraic variety} we mean an integral, separated scheme over $\mathbb{C}$. We will often think of a variety $X$ as the associated complex analytic space $X^\text{an}$. Therefore, when employing analytic notions within the realm of algebraic geometry we actually refer to $X^\text{an}$ (e.g. \emph{manifold}, \emph{analytic neighbourhoods}, \emph{holomorphic forms} etc.).  

Given a variety, we denote by $X^\text{reg}$ and $X^\text{sing}$ the open and closed subsets given by the smooth and singular locus (with its reduced structure), respectively. When $X$ is normal, and $j\colon X^\text{reg} \hookrightarrow X$ denotes the open immersion of its smooth locus, the sheaf of reflexive $p$-forms (for $0\leq p\leq \dim X$) is defined to be
\begin{equation}
\Omega^{[p]}_X \coloneqq j_\ast \Omega^p_{X^\text{reg}}\cong (\Lambda^p \Omega_X)^{\ast \ast},
\end{equation}
so that $H^0(X, \Omega_X^{[p]})=H^0(X^\text{reg}, \Omega_{X^\text{reg}}^p)$, i.e. a reflexive $p$-form is nothing but a regular $p$-form on the smooth locus $X^\text{reg}$. We will denote by $\omega_X\coloneqq \Omega_X^{[\dim X]}$ and call it the \emph{canonical sheaf} of $X$. When $X$ is Cohen--Macaulay, it agrees with the dualising sheaf, which is a line bundle when $X$ is Gorenstein.

\subsection*{Acknowledgments} We would like to thank Richard Thomas for being enthusiastic about this project and for reading several drafts of this paper. The second author is especially grateful for his constant support and guidance. Additionally, the second author wishes to thank Arend Bayer, Daniel Huybrechts and Travis Schedler for stimulating discussions and valuable feedback, as well as Federico Bongiorno, Alessio Bottini, Luigi Martinelli, and Mirko Mauri for very helpful conversations.

\addtocontents{toc}{\protect\setcounter{tocdepth}{1}}
\tableofcontents

\section{Holomorphic symplectic varieties}\label{section_symplectic_varieties}

\subsection{Holomorphic symplectic varieties} 

Holomorphic symplectic manifolds are Calabi--Yau manifolds admitting a holomorphic, everywhere non-degenerate two-form. Among them, the class of \emph{irreducible holomorphic symplectic manifolds}, i.e. compact K\"ahler simply connected holomorphic symplectic manifolds $X$ such that $H^{2,0}(X)$ is one-dimensional, has been extensively studied over the last few decades as they are building blocks for compact K\"ahler manifolds with trivial first Chern class, according to the Beauville--Bogomolov decomposition Theorem \cite[Thm. 2]{Beauville-Chern}. Despite the richness of their geometry, there is a severe lack of examples, which has led to a surge in research efforts to uncover new instances. Up to now, we only know two families in each even dimension \cite{Beauville-Chern}, and two exceptional families in dimension 10 and 6 \cite{O'Grady10, O'Grady6}, and they all arise from moduli spaces of sheaves on symplectic surfaces. 

In an attempt to generalise such a notion to the singular setting, Beauville introduced the class of \emph{holomorphic symplectic varieties}:

\begin{defn}[Holomorphic symplectic varieties, {\cite[Def. 1.1]{Beauville-symplectic-singularities}.}] A \emph{holomorphic symplectic structure} on a normal variety $X$ is a closed reflexive two-form $\omega\in H^0(X, \Omega_X^{[2]})$ which is non-degenerate at every point of $X_\text{reg}$. We say $(X,\omega)$ is a \emph{holomorphic symplectic variety} if for every resolution of singularities $f\colon \tilde{X}\to X$ the symplectic form $\omega\vert_{X_\text{reg}}$ extends to a holomorphic two-form on $\tilde{X}$.
\end{defn} 

\begin{rmk}[Symplectic singularities are canonical] \label{pull-back-morphism} By \cite[Thm. 6]{Namikawa}, a normal variety with a holomorphic symplectic structure is a holomorphic symplectic variety if and only if it is has rational Gorenstein singularities. In particular, by Kempf's criterion \cite[p. 50]{toroidal-embeddings} a holomorphic symplectic variety $X$ has canonical singularities, so that for every resolution of singularities $f\colon \tilde{X}\to X$ there exists a pull-back injective morphism
\begin{equation}\label{pull-back-morphism-eq}
f^\ast \omega_X \to \omega_{\tilde{X}},
\end{equation}
which extends the corresponding isomorphism over the smooth locus. 
\end{rmk}

\begin{ex} \label{examples-HSV} \begin{enumerate}
\item  Projective cones over canonical curves are singular symplectic surfaces: they can be regarded as one-point compactifications of the cotangent bundle obtained by collapsing the section at infinity of the natural projective completion. 
\item If $X$ is a symplectic quasi-projective variety and $G$ is a finite group of symplectic automorphisms of $X$, the normalisation of the irreducible components of the fixed locus of $G$ and the quotient $X/G$ are symplectic varieties \cite{Fujiki, Beauville-symplectic-singularities}.
\item The main source of compact examples is provided by moduli spaces of sheaves on symplectic surfaces, after Mukai's seminal work \cite{Mukai, Mukai-bundlesK3}.
\end{enumerate}
\end{ex}

A holomorphic symplectic variety $X$ always admits a finite stratification by \emph{symplectic leaves}, i.e. locally closed smooth symplectic subvarieties: 

\begin{prop}[Stratification by symplectic leaves, {\cite{Kaledin-symplectic-Poisson, Schedler-Kaplan}}] \label{stratification-symplectic-leaves} Let $X$ be a holomorphic symplectic variety. There is a finite filtration by reduced closed subschemes
\begin{equation}
X=X_0 \supseteq X_1 \supseteq \cdots \supseteq X_k
\end{equation}
inducing a stratification of $X$ with the following properties:
\begin{enumerate}[label={\upshape(\roman*)}]
\item $X_{i+1}=X_i^{\textnormal{sing}}$;
\item Each smooth stratum $X_{i}\setminus X_{i+1}\subseteq X$ is a symplectic subvariety;
\item $X$ is locally a product along each stratum: for all $x\in X_{i}\setminus X_{i+1}$ there exists a germ $(Z,0)$ of a holomorphic symplectic variety and an isomorphism of analytic germs of symplectic varieties
\begin{equation}\label{product_decomposition}
(X,x)\cong (X_{i}\setminus X_{i+1}, x)\times (Z,0).
\end{equation} 
\end{enumerate}
\end{prop}

Note that the last condition implies that any two points in a connected component of an open stratum have isomorphic analytic neighbourhoods in $X$. In other words, such connected components are the isosingular strata of $X$ \cite[Def. 0.1]{Ephraim}. 

\begin{rmk}[Normal flatness]\label{normal_flatness} Recall an immersion of schemes $i\colon Y\to X$ is normally flat if the normal cone $C_{Y/X}$ is flat over $Y$ \cite[App. III]{Hermann}. The product decomposition \eqref{product_decomposition} implies \textit{a fortiori} a product decomposition of the normal cone, so that $X$ is normally flat along each stratum \cite[App. III, Prop. 1.4.12]{Hermann}.
\end{rmk}

\subsection{Symplectic and crepant resolutions} Among the resolutions of singularities of a symplectic variety, a pivotal role is played by those which are themselves symplectic: 

\begin{defn}[Symplectic and crepant resolutions, {\cite[Def. 3.1]{Kaledin-crepant}}] If $(X,\omega)$ is a holomorphic symplectic variety, we say a projective resolution of singularities $f\colon \tilde{X}\to X$ is \emph{symplectic} if the symplectic form $\omega\vert_{X_\text{reg}}$ extends to a holomorphic symplectic two-form on $\tilde{X}$.

We say a projective resolution of singularities $f\colon \tilde{X}\to X$ is \emph{crepant} if the canonical morphism $f^\ast \omega_X \to \omega_{\tilde{X}}$ from Remark \ref{pull-back-morphism} is an isomorphism. 
\end{defn}

Note that the injective morphism $f^\ast \omega_X\to \omega_{\tilde{X}}$ defines an effective Cartier divisor $E$ on $\tilde{X}$, which describes the locus where \eqref{pull-back-morphism-eq} fails to be an isomorphism. We call the closed subset $f(E)\subseteq X$ the \emph{discrepancy centre} of $f$ \cite[\textsection 2.3]{Kollar-Mori}.

Clearly every symplectic resolution $f\colon \tilde{X} \to X$ is crepant, and hence in particular $\omega_{\tilde{X}}\cong \OO{\tilde{X}}$. The converse is also true: if $f\colon \tilde{X}\to X$ is crepant, then the extension of $\Lambda^{\dim X} \omega\vert_{X_\text{reg}}$ to $\tilde{X}$ has no zeros, and hence $\omega \vert_{X_\text{reg}}$ extends to a symplectic form on $\tilde{X}$ \cite[Prop. 3.2]{Kaledin-crepant}.

\begin{rmk}[Trivial canonical bundle implies crepant] \label{trivial-omega-crepant} Note that if $f\colon \tilde{X}\to X$ is a resolution of singularities of a holomorphic symplectic variety, it is enough to have \emph{any} isomorphism $\omega_{\tilde{X}}\cong \OO{\tilde{X}}$ in order for $f$ to be crepant. Indeed, the map $f^\ast \omega_X\to \omega_{\tilde{X}}$ is simply a function on $\tilde{X}$ which is non-zero over the smooth locus of $X$. By normality, it extends to a non-zero function on $X$ and hence it is non-zero on $\tilde{X}$ as well. 
\end{rmk}

\begin{rmk}[Being a crepant resolution is local over the base] \label{crepancy-local} Note that given a resolution of singularities $f\colon \tilde{X}\to X$ of a holomorphic symplectic variety, being crepant is local -- both Zariski and analytically -- over the base, and by Remark \ref{trivial-omega-crepant} it is equivalent to $\omega_{\tilde{X}}$ being locally trivial over $X$. 

If now $f\colon \tilde{X}\to X$ is \emph{any} projective birational morphism, it follows that being a crepant resolution is also local over the base. Indeed, if $f$ is a crepant resolution around every point of $X$, it follows in particular that $\tilde{X}$ is smooth, so that the morphism $f^\ast \omega_X\to \omega_{\tilde{X}}$ is well-defined and is locally -- hence globally -- an isomorphism.
\end{rmk}

\section{The local model and its resolution}\label{section_local_model}

In this Section we study the symplectic geometry of a specific class of nilpotent orbits closures $\mathcal{N}_{\leq r}(V)$ of rank $r\geq 0$ in the symplectic Lie algebra $\splie(V)$ of a symplectic vector space $(V,\omega)$, showing they admit a symplectic resolution given by the blow-up of the reduced singular locus when $2r=\dim V$ (Theorem \ref{blow-up-nilpotent-cone}).

\subsection{Symplectic nilpotent orbits closures are symplectic varieties}

Let $(V,\omega)$ be a $2n$-dimensional complex symplectic vector space. We denote by $\SP(V)\subseteq\Aut (V)$ the group of linear automorphisms of $V$ preserving the symplectic form $\omega$, i.e. $A\in \Aut (V)$ belongs to $\SP (V)$ if and only if 
\begin{equation}
\omega(Av, Aw)= \omega(v,w),  \text{ for all } v,w\in V.
\end{equation}
Similarly, we denote by $\splie(V)\subseteq\End (V)$ the associated Lie algebra, i.e. the set of endomorphisms $B\in \End(V)$ satisfying 
\begin{equation}\label{splie-condition}
\omega(Bv, w)= -\omega(v,Bw),  \text{ for all } v,w\in V.
\end{equation}

We identify $\splie(V)$ with its dual Lie algebra $\splie(V)^\ast$ by means of the Killing form, so that we can regard adjoint orbits -- which in this case are simply conjugacy classes under the action of $\SP(V)$ -- as holomorphic symplectic manifolds via the Kirillov--Kostant--Souriau symplectic form (see \cite[Ch. 1]{Ginzburg}). \\

For any integer $0\leq r\leq n$ we will be interested in the subvariety of $\splie(V)$ 
\begin{equation}
\mathcal{N}_r(V) \coloneqq \{ B \in \splie(V) \mid B^2=0 \text{ and } \operatorname{rk}B=r\}
\end{equation}
given by 2-nilpotent elements of rank $r$. It is an easy exercise in linear algebra to see that $\mathcal{N}_r(V)$ is a single adjoint orbit and hence a holomorphic symplectic manifold. It is locally closed in $\splie(V)$, and its closure is given by  
\begin{equation}\label{2-nilpotent-orbits-closure-defn}
\mathcal{N}_{\leq r}(V) \coloneqq \overline{\mathcal{N}}_{r}(V) = \bigcup_{r' \leq r} \mathcal{N}_{r'}(V).
\end{equation}

The closure $\mathcal{N}_{\leq r}(V)$ of the 2-nilpotent orbit of rank $r$ is a $\mathbb{Q}$-factorial symplectic variety of dimension
\begin{equation}
\dim \mathcal{N}_{\leq r}(V) = r(2n-r+1),
\end{equation} 
whose singular locus is given by $\mathcal{N}_{\leq r-1}(V)\subseteq \mathcal{N}_{\leq r}(V)$ \cite{Kraft-Procesi, Panyushev}.

\subsection{Isotropic and Lagrangian Grassmannians} 
Let $(V,\omega)$ be a $2n$-dimensional complex symplectic vector space. For a linear subspace $W\subseteq V$ we denote by $W^\perp \subseteq V$ its orthogonal with respect to $\omega$, i.e.
\begin{equation}
W^\perp \coloneqq \{v\in V \mid \omega(v,w)=0 \textnormal{ for all } w\in W \}. 
\end{equation}

Recall a subspace $W\subseteq V$ is \emph{isotropic} if $W\subseteq W^\perp$ -- equivalently $\omega$ vanishes on $W$.  Any isotropic subspace $W$ has $\dim W\leq n$. We say an isotropic subspace $W\subseteq V$ is \emph{Lagrangian} if it has maximal dimension $n$, so that $W=W^\perp$. \\

Given $0< r \leq 2n$ we denote by $\Grass(r,V)$ the ordinary Grassmannian of $r$-dimensional subspaces of $V$ and by $\Lambda(r,V)\subseteq \Grass(r,V)$ the Grassmannian of isotropic $r$-dimensional subspaces of $V$. It is a smooth submanifold of $\Grass(r,V)$ of dimension
\begin{equation}\label{dimension-isotropic-grass}
\dim \Lambda(k,V) = 2nr-\frac{3r^2-r}{2}.
\end{equation}
Note that $\Lambda(n,V)$ is the Lagrangian Grassmannian and $\Lambda(r,V)$ is empty for $r>n$.  We denote by $\mathcal{U}$ (resp. $\mathcal{Q}$) the universal subbundle (resp. quotient bundle) of $V\otimes \OO{\Grass(r,V)}$ and we use the same notation for their restriction to $\Lambda(r,V)\subseteq \Grass(r,V)$. \\

Since $\Lambda(r,V)\subseteq \Grass(r,V)$ is cut out by a regular section of $\Lambda^2 \mathcal{U}^\ast$, given by the restriction of the symplectic form $\omega$ to the universal subbundle $\mathcal{U}$, we get an exact sequence
\begin{equation}
\begin{tikzcd}[column sep=small]
0 \ar[r] & \Lambda^2 \mathcal{U}  \ar[r] & \mathcal{U} \otimes \mathcal{Q}^\ast \ar[r] & \Omega_{\Lambda(r,V)}\ar[r] & 0
\end{tikzcd}
\end{equation} 
In the Lagrangian case, we have $\mathcal{U}=\mathcal{U}^\perp$. Hence, the identification $\mathcal{Q}^\ast \cong \mathcal{U}^\perp$ induced by the symplectic form gives $\Omega_{\Lambda(n,V)}  \cong \operatorname{Sym}^2 \mathcal{U}$.

\subsection{Resolutions of symplectic 2-nilpotent orbits} By means of the universal subbundle of the Grassmannian of $r$-dimensional isotropic subspaces of $V$, one can write down an explicit resolution of singularities of the closure $\mathcal{N}_{\leq r}(V)$:

\begin{lem}[Resolution of symplectic 2-nilpotent orbits]\label{resolution-2nilpotent-cone} Let $0< r\leq n$ be an integer. The closure $\mathcal{N}_{\leq r}(V)$ of the 2-nilpotent orbit of rank $r$ in $\splie(V)$ admits a resolution of singularities given by the total space of $\Sym^2 \mathcal{U}$, where $\mathcal{U}\subseteq V\otimes \OO{\Lambda(r,V)}$ is the universal subbundle of the isotropic Grassmannian $\Lambda(r,V)$. 
\end{lem}

\begin{proof} One considers the incidence variety
\begin{equation}\label{total-space-resolution-nilpotent-orbits}
\tilde{\mathcal{N}}_{\leq r}(V) \coloneqq \left\lbrace (B, W)\in \mathcal{N}_{\leq r}(V)\times \Lambda(r,V) \mid \im B \subseteq W  \right\rbrace,
\end{equation}
with projections $f$ and $g$ onto the two factors. Note that since $W$ is isotropic and $B\in \splie(V)$, one has $\ker B= (\im B)^\perp$, so that the condition $\im B \subseteq W$ is equivalent to 
\begin{equation}\label{condition-equivalent-resolution}
\im B \subseteq W  \subseteq W^\perp \subseteq \ker B. 
\end{equation}
The projection $f$ is one-to-one over the rank $r$ locus, as $\im B$ is itself in $\Lambda(r,V)$. 

Now it is easy to see that  $g\colon  \tilde{\mathcal{N}}_{\leq r}(V) \to \Lambda(r,V)$ is given by the total space of the bundle $\Sym^2\mathcal{U}$. Indeed, note that by \eqref{condition-equivalent-resolution} there is a natural closed immersion
\begin{equation}\label{subbundle-resolution}
\begin{split}
\tilde{\mathcal{N}}_{\leq r}(V) & \to \operatorname{Tot}\left(\mathcal{H}om(\mathcal{U}^\ast, \mathcal{U})\right) \\
(B,W) & \mapsto \tilde{B}\colon V/W^\perp \cong W^\ast\to W,
\end{split}
\end{equation}
where 
\begin{equation}
\begin{split}
\tilde{B}\colon W^\ast & \to W \\
f=\omega(v,-) & \mapsto Bv.
\end{split}
\end{equation}
Now such a $\tilde{B}\in \Hom(W^\ast, W)\cong W\otimes W$ lies in $\Sym^2 W$ if and only if for all $f,g\in W^\ast$ one has $g(\tilde{B}(f))=f(\tilde{B}(g))$. Writing $f=\omega(v,-)$ and $g=\omega(w, -)$ for suitable $v,w\in V$, this is equivalent to
\begin{equation}
\omega(w,Bv)=\omega(v,Bw)=-\omega(Bw, v),
\end{equation}
which is equivalent to $B\in \splie(V)$. Hence we deduce that $\tilde{\mathcal{N}}_{\leq r}(V)$ and $\Sym^2 \mathcal{U}$ agree as subbundles of $\HHom(\mathcal{U}^\ast, \mathcal{U})\cong \mathcal{U}\otimes \mathcal{U}$ via \eqref{subbundle-resolution}.
\end{proof}

Now we want to understand in which cases the resolution given by Lemma \ref{resolution-2nilpotent-cone} is crepant, and provide an intrinsic description of it. \\

We will need a technical result first. Recall that given a morphism of schemes $f\colon X\to Y$, we say a line bundle $\xi \in \Pic(X)$ is $f$-\emph{very ample} if there is an immersion $i\colon X\hookrightarrow \mathbb{P}^m_Y$ over $Y$ such that $\xi \cong i^\ast \OO{\mathbb{P}^m_Y}(1)$ \cite[\textsection II.5]{Hartshorne}. 

\begin{prop}\label{curious-property-blow-up} Let $f\colon X \to Y$ be a projective birational morphism of algebraic varieties. Assume there is a closed subscheme $Z\subseteq Y$ such that the inverse image ideal sheaf $f^{-1}\mathcal{I}_Z \cdot \OO{X}$ is an invertible sheaf which is $f$-very ample and the natural map
\begin{equation}\label{push-of-ideal-exc-divisor}
\mathcal{I}_Z \to f_\ast \left( f^{-1}\mathcal{I}_Z \cdot \OO{X} \right)
\end{equation}
is an isomorphism. Then the birational morphism
\begin{equation}\label{universal-property}
\begin{tikzcd}[column sep=small]
X \ar[rr]\ar[dr] & & \operatorname{Bl}_Z Y\ar[dl] \\
& Y &
\end{tikzcd}
\end{equation}
induced by the universal property is an isomorphism. 
\end{prop}

\begin{proof} This is a variant of the proof of \cite[Thm. II.7.17]{Hartshorne}. As in  \textit{ibid.}, our hypotheses ensure that $X \cong \mathbf{Proj}_Y \bigoplus_{k\geq 0} f_\ast \mathcal{I}_E^k, $ where $\mathcal{I}_E$ is the ideal sheaf of the scheme-theoretic inverse image of $Z$. Now the $f$-ampleness of $\mathcal{I}_E$, together with the assumption \eqref{push-of-ideal-exc-divisor} guarantee that, for $k\gg 0$, the natural map $\mathcal{I}_Z^k\to f_\ast \mathcal{I}_E^k$ is surjective. Hence, the induced morphism \eqref{universal-property} is a birational closed immersion. Since $Y$, and hence $\operatorname{Bl}_Z Y$, is integral, it is an isomorphism. 
\end{proof}

Now we are finally ready to state the main result about the existence of symplectic resolutions for symplectic 2-nilpotent orbits:

\begin{thm}[Symplectic resolution of symplectic 2-nilpotent orbits]\label{blow-up-nilpotent-cone} Let $0< r\leq n$ be an integer. 
When $r<n$ the closure $\mathcal{N}_{\leq r}(V)$ of the 2-nilpotent orbit of rank $r$ in $\splie(V)$ does not admit any symplectic resolution.

When $r=n$, the 2-nilpotent cone $\mathcal{N}_{\leq n}(V)$ admits a symplectic resolution given by the total space of the cotangent sheaf of the Lagrangian Grassmannian $\Lambda(n, V)$. Moreover such a resolution is isomorphic to the blow-up of $\mathcal{N}_{\leq n}(V)$ along the reduced singular locus $\mathcal{N}_{\leq n-1}(V)$.
\end{thm}

\begin{proof}The first part follows immediately from a standard argument relying on the fact that the singular locus has codimension $>2$. Indeed, $\mathbb{Q}$-factoriality implies that the exceptional locus of any birational morphism to $\mathcal{N}_{\leq r}(V)$ must be pure of codimension one \cite[\textsection 1.40]{Debarre}, but symplectic resolutions are semi-small \cite[Lemma 2.11]{Kaledin-symplectic-Poisson}, so the singular locus must have codimension $\leq 2$. 

When $r=n$, the resolution
\begin{equation} \label{resolution-nilpotent-orbit-max-rank}
f\colon \tilde{\mathcal{N}}_{\leq n}(V) \to \mathcal{N}_{\leq n}(V)
\end{equation}
provided by \eqref{total-space-resolution-nilpotent-orbits} in Lemma \ref{resolution-2nilpotent-cone} actually coincides with the total space of the cotangent bundle of $\Lambda(n,V)$ and hence it is crepant by Remark \ref{trivial-omega-crepant}.

For the last statement, we want to make use of Proposition \ref{curious-property-blow-up}. Let $E\subseteq \tilde{\mathcal{N}}_{\leq n}(V)$ be the effective Cartier divisor given by the inverse image of the reduced singular locus $\mathcal{N}_{\leq n-1}(V)$ of $\mathcal{N}_{\leq n}(V)$. Since $\tilde{\mathcal{N}}_{\leq n}(V)$ is a vector bundle over the Lagrangian Grassmannian, the associated line bundle is given by
\begin{equation}
\OO{\tilde{\mathcal{N}}_{\leq n}(V)}(E) = g^\ast \OO{\Lambda(n,V)}(q),
\end{equation}
for some $q\in \mathbb{Z}$, where $\OO{\Lambda(n,V)}(1)$ is the Pl\"ucker line bundle on the Grassmannian $\Lambda(n,V)$. 

Note that by adjunction, since $\tilde{\mathcal{N}}_{\leq n}(V)$ has trivial canonical bundle, the dualising sheaf $\omega_E$ is given by $\OO{E}(E)=g^\ast \OO{\Lambda(n,V)}(q)\otimes \OO{E}$. Now consider the restriction $f\colon E \to \mathcal{N}_{\leq n-1}(V)$: it is a $\mathbb{P}^1$-bundle over the smooth locus $\mathcal{N}_{n-1}(V)$. Given $B\in \mathcal{N}_{ n-1}(V)$, once again by adjunction we get
\begin{equation}\label{canonical-fibre}
g^\ast \OO{\Lambda(n,V)}(q) \otimes \OO{f^{-1}(B)}= \omega_E \otimes \OO{f^{-1}(B)} = \omega_{f^{-1}(B)}.
\end{equation}

On the other hand, the fibre 
\begin{equation}
f^{-1}(B) = \{B\}\times   \{W \in \Lambda(n,V) \mid \operatorname{Im}B \subseteq W \subseteq \ker B\} \cong \mathbb{P}(\ker B/\operatorname{Im} B),
\end{equation}
is a linear subspace contained in $\Lambda(n,V)$ under the Pl\"ucker embedding, so that \eqref{canonical-fibre} implies $q=-2$ and hence the ideal sheaf of $E$ is $f$-very ample. 

Finally, since $\mathcal{N}_{\leq n-1}(V)$ and $\mathcal{N}_{\leq n}(V)$ are both normal, by Zariski's Main Theorem and Stein factorisation \cite[Cor. III.11.4-5]{Hartshorne} the natural maps
\begin{equation}
\OO{\mathcal{N}_{\leq n}(V)}\to f_\ast \OO{\tilde{\mathcal{N}}_{\leq n}(V)}, \quad \OO{\mathcal{N}_{\leq n-1}(V)}\to f_\ast \OO{E}
\end{equation}
are both isomorphisms. It follows that the map
\begin{equation}
 \mathcal{I}_{\mathcal{N}_{\leq n-1}(V)} \to f_\ast \mathcal{I}_E = f_\ast \left(f^{-1} \mathcal{I}_{\mathcal{N}_{\leq n-1}(V)} \cdot \OO{\tilde{\mathcal{N}}_{\leq n}(V)} \right)
\end{equation}
is also an isomorphism. 
\end{proof}

\section{Equivariant bundles on hyperelliptic curves}\label{section_equivariant_hyperelliptic}

In this Section we describe the action induced by the hyperelliptic involution on the moduli space of rank two bundles on a smooth hyperelliptic curve. 

\subsection{Equivariant sheaves} Assume $G$ is a finite group acting on a projective variety $X$. Recall in this setting a $G$-\emph{equivariant structure} on a sheaf $F$ on $X$ is a collection $\alpha\coloneqq \{\alpha_g\}_{g\in G}$ of morphisms of sheaves $\alpha_g \colon g^\ast F \to F$ for all $g\in G$ such that
\begin{enumerate}[label={\upshape(\roman*)}]
\item $\alpha_e=\ide_F$;
\item $\alpha_{g\cdot h}= \alpha_h \circ h^\ast \alpha_g$ for all $g,h\in G.$
\end{enumerate}
We say $(F,\alpha)$ is a $G$-\emph{equivariant sheaf} on $X$. 

When $F$ is locally free, the datum of a $G$-equivariant structure is the same as a lift of the $G$-action by bundle automorphisms to the total space $\operatorname{Tot}(F)$. When $F$ is an invertible sheaf, we often use $G$-\emph{linearisation} instead of $G$-equivariant structure. \\

Morphisms of $G$-equivariant sheaves are morphisms of sheaves which commute with the equivariant structure: we will denote by $\Coh^G(X)$ the category of $G$-equivariant coherent sheaves on $X$ -- all morphisms are required to be $\OO{X}$-linear here. It is an abelian category with enough injectives, equipped with internal hom and tensor product \cite[\textsection 5]{Tohoku}. Moreover, if $X$ and $Y$ are quasi-projective varieties equipped with a $G$-action and $f\colon X\to Y$ is a $G$-equivariant morphism, the pull-back and push-forward functors (assuming they are defined) lift to the equivariant categories
\begin{equation}
f^\ast \colon \Coh^G(Y) \to \Coh^G(X), \quad f_\ast \colon \Coh^G(X)\to \Coh^G(Y).
\end{equation}
In particular, by means of equivariant push-forward to the point we get an induced $G$-action on the cohomology groups $H^i(X, F)$ of a $G$-equivariant sheaf $(F,\alpha)$. Moreover, when the action of $G$ on $Y$ is trivial, taking $G$-invariants defines a subsheaf $f_\ast^G F\subseteq f_\ast F$.

\subsection{Hyperelliptic curves}  Let $C$ be a smooth curve of genus $g\geq 2$. Recall we say $C$ is \emph{hyperelliptic} if it is endowed with a two-to-one morphism 
\begin{equation}\label{hyperelliptic-map}
f\colon C\to \mathbb{P}^1,
\end{equation}
ramified along a set of $2g+2$ distinct points, defining the \emph{ramification divisor} $R\in \Div(C)$, whose image via $f$ is the \emph{branch divisor} $B\in \Div(\mathbb{P}^1)$. \\

Let $s\in H^0(\mathbb{P}^1, \OO{\mathbb{P}^1}(2g+2))$ be a section cutting out $B$ and let $L\coloneqq \OO{\mathbb{P}^1}(g+1)$ be the square-root of $\OO{\mathbb{P}^1}(2g+2)$.  Let $\pi\colon \operatorname{Tot}(L)\to C$ be the total space of $L$: the curve $C$ can be recovered uniquely up to isomorphism via the cyclic covering construction as the zero locus of the section
\begin{equation}
\lambda^2-\pi^\ast s \in H^0(\operatorname{Tot}(L), \pi^\ast L^2),
\end{equation}
where $\lambda\in H^0(\operatorname{Tot}(L), \pi^\ast L)$ is the tautological section. \\

The map \eqref{hyperelliptic-map} is then simply given by the restriction of $\pi$ to $C\subseteq \operatorname{Tot}(L)$ and we have  
\begin{equation}\label{Hurwitz-formula}
\begin{split}
f_\ast \OO{C} & \cong \OO{\mathbb{P}^1}\oplus L^{-1}, \\
\omega_C & \cong f^\ast \omega_{\mathbb{P}^1} \otimes \OO{C}(R).
\end{split}
\end{equation}

The involution on $\operatorname{Tot}(L)$ given by multiplication by $-1$ on the fibres of $L$ induces an involution $\sigma \in \Aut(C)$, which we will call the \emph{hyperelliptic involution}, which swaps the points in the fibres of $f$ and such that $C/ \langle \sigma \rangle \cong \mathbb{P}^1$ via $f$. 

\subsection{Equivariant bundles on hyperelliptic curves} We will be interested in the $\Ztwo$-action on $C$ defined by $\sigma\in \Aut(C)$. 
We start by describing the induced action on the Picard group:

\begin{lem}[Hyperelliptic involution on $\Pic(C)$]\label{action-on-Pic} For every line bundle $\xi \in \Pic(C)$ one has an isomorphism
\begin{equation}\label{formula-hyperelliptic-line-bundles}
\sigma^\ast \xi \cong \xi^{-1}\otimes (f^\ast \OO{\mathbb{P}^1}(1))^{\deg (\xi)}.
\end{equation}
In particular every $\xi \in \Pic^0(C)$ satisfies $\sigma^\ast \xi\cong \xi^{-1}$. Moreover, if $\eta\in \Pic^{2d}(C)$ satisfies $\eta^2\cong f^\ast \OO{\mathbb{P}^1}(d)$ for some $d\in \mathbb{Z}$, then $\sigma^\ast \eta\cong \eta$.
\end{lem}

\begin{proof} We can assume $H^0(C, \xi)\neq 0$ -- otherwise write $\xi= \xi_1\otimes \xi_2$ with $H^0(C, \xi_1)$ and $H^0(C, \xi_2^{-1})$ both non-zero and use additivity of \eqref{formula-hyperelliptic-line-bundles}.  We can then induct on $\deg(\xi)$, so it is enough to prove the claim for $\xi=\OO{C}(p)$ for a point $p\in C$. In this case, one has $p+\sigma^\ast p=f^\ast (f(p))\in \operatorname{Div}(C)$, so that
\begin{equation}
\OO{C}(p)\otimes \sigma^\ast \OO{C}(p)\cong f^\ast \OO{\mathbb{P}^1}(1). \qedhere
\end{equation}
\end{proof}

\begin{ex}[Equivariant structures on $\omega_{C}$]\label{equiv-omega-hyper} The natural equivariant structure on the canonical bundle $\omega_C$ induced by the functorial isomorphism $\sigma^\ast \Omega_C\cong \Omega_C$ induces a $(-1)$-action on the fibres over fixed points in $R$, and the corresponding action on global sections is given by $-\ide$. Whenever we consider $\omega_C$ as an equivariant bundle, we will consider it with this structure.

On the other hand, by \eqref{Hurwitz-formula}, $\omega_C\cong f^\ast \OO{\mathbb{P}^1}(g-1)$ is a pull-back, so it has a trivial $\Ztwo$-equivariant structure inducing the identity on the fixed fibres and the trivial action on global sections. 

In terms of the action on the total space, the two differ by the involution given by multiplication by $-1$ on the fibres of $\omega_C$, which corresponds to tensoring with the non-trivial equivariant structure on the structure sheaf. \end{ex}

Note that if $(E,\alpha)$ is an equivariant vector bundle on $C$, the cohomology groups $H^i(C, E)$ are endowed with a $G$-action, so that we can define the quantity
\begin{equation}
\chi^{\Ztwo}(E) \coloneqq \dim H^0(C,E)^{\Ztwo} - \dim H^1(C, E)^{\Ztwo}. 
\end{equation}

We have the following equivariant version of Riemann--Roch, which will allow us to compute the dimension of the fixed locus of the hyperelliptic  involution in the moduli space of bundles.

\begin{prop}[Equivariant Riemann--Roch {\cite[Prop. 2.2]{Desale-Ramanan}}]\label{equivariant-RR} Let $(E,\alpha)$ be an equivariant vector bundle on $C$. For $p\in R$, denote by $E_p^{\Ztwo}\subseteq E_p$ the $\alpha_p$-invariant subspace of the fibre $E_p$. Then
\begin{equation}
\chi^{\Ztwo}(E) = \frac{1}{2}\deg(E)-\rk(E)g+\frac{1}{2}\sum_{p\in R} \dim E_p^{\Ztwo}.
\end{equation}
\end{prop}

\subsection{Moduli spaces of vector bundles} 
Let $C$ be a smooth curve of genus $g\geq 2$. For $(r,d)\in \mathbb{Z}_{>0}\times \mathbb{Z}$ we denote by $\mathcal{U}_C(r,d)$ the moduli space parametrising $S$-equivalence classes of slope-semistable vector bundles on $C$ with rank $r$ and degree $d$ \cite{Seshadri-unitary, Narasimhan-Ramanan, Seshadri}. It is a normal, irreducible, projective variety of dimension
\begin{equation}
\dim \mathcal{U}_C(r,d)= r^2(g-1)+1.
\end{equation}
We denote by  $\mathcal{U}^s_C(r,d)\subseteq \mathcal{U}_C(r,d)$ the smooth open subset parametrising isomorphism classes of stable vector bundles. 

Having fixed a degree $d$ line bundle $\xi\in \Pic^d(C)$, we can consider the fibre of the isotrivial determinant morphism
\begin{equation}
\det \colon \mathcal{U}_C(r,d) \to \Pic^d(C),
\end{equation}
which we will denote by $\mathcal{SU}_C(r,\xi)$. It parametrises $S$-equivalence classes of slope-semistable vector bundles of rank $r$ and fixed determinant $\xi$  and is a normal, irreducible, projective variety of dimension
\begin{equation}
\dim \mathcal{SU}_C(r,d)= (r^2-1)(2g-2).
\end{equation}
It contains a smooth open subset $\mathcal{SU}^s_C(r,\xi)\subseteq \mathcal{SU}_C(r,\xi)$ parametrising isomorphism classes of stable vector bundles \cite[\textsection 1.VI]{Seshadri}. \\

When the integers $r$ and $d$ are coprime, every semistable bundle is stable, so that both $\mathcal{U}_C(r,d)$ and $\mathcal{SU}_C(r,\xi)$ are smooth and projective \cite[Thm. 1.III.17]{Seshadri}. If $r$ and $d$ are not coprime, and $g\geq 3$ or $g=2$ and $r\geq 3$, then the smooth loci of $\mathcal{U}_C(r,d)$ and $\mathcal{SU}_{C}(r,\xi)$ are precisely the stable loci \cite[Thm. 1.V.45]{Seshadri}. When $g=2$, the moduli spaces $\mathcal{U}_C(2,0)$ and $\mathcal{SU}_C(2, \OO{C})$ are actually smooth -- the latter is isomorphic to $\mathbb{P}^3$ via the theta-map \cite{Narasimhan-Ramanan}. 

\begin{rmk}[Local triviality of determinant morphism]\label{triviality-determinant-morphism} Having fixed a line bundle $\xi\in \Pic^d(C)$, the determinant morphism $\det\colon \mathcal{U}_C(r,d)\to \Pic^d(C)$ becomes trivial after a finite \'etale cover. Indeed, we have the following cartesian diagram
\begin{equation} 
\begin{tikzcd}[column sep=2cm]
\mathcal{SU}_C(r, \xi)\times \Pic^d(C) \ar[r, "(- \otimes -) \otimes \xi^{-1}"] \ar[d] &  \mathcal{U}_C(r, d) \ar[d, "\det"] \\
\Pic^d(C) \ar[r, "(-)^r \otimes \xi^{1-r}"] & \Pic^d(C)
\end{tikzcd}
\end{equation}
where the horizontal arrows are Galois covers with Galois group $\Pic^0(C)[r]$. In particular, the moduli spaces $\mathcal{SU}_C(r, \xi)$ and $\mathcal{U}_C(r,d)$ are stably isosingular as in \cite[Def. 2.6]{Mirko}. 
\end{rmk}

We will be specifically interested in moduli spaces of rank two bundles on $C$, so we describe the structure of their singular locus explicitly:

\begin{prop}[Singularities of $\mathcal{U}_C(2,0)$ and $\mathcal{SU}_C(2, \OO{C})$, {\cite[Thm. 1.V.45]{Seshadri}}]\label{singularities-U-SU} Let $C$ be a smooth curve of genus $g\geq 3$. We have the following stratification of $\mathcal{U}_C(2,0)$:
\begin{equation}
\mathcal{U}_C(2,0) \supseteq \Sigma(\mathcal{U}_C(2,0)) \supseteq \Omega(\mathcal{U}_C(2,0)),
\end{equation}
where
\begin{equation}
\begin{split}
\Sigma(\mathcal{U}_C(2,0)) & \coloneqq \mathcal{U}_C(2,0)^{\textnormal{sing}}= \{\xi \oplus \xi' \mid \xi, \xi' \in \Pic^0(C) \}, \\
\Omega(\mathcal{U}_C(2,0))& \coloneqq  \Sigma(\mathcal{U}_C(2,0))^{\textnormal{sing}} =\{ \xi\oplus \xi \mid \xi \in \Pic^0(C)\},
\end{split}
\end{equation}
where we consider bundles up to $S$-equivalence. 

We have an analogous stratification for $\mathcal{SU}_C(2, \OO{C})$:
\begin{equation}
\mathcal{SU}_C(2, \OO{C}) \supseteq \Sigma(\mathcal{SU}_C(2, \OO{C})) \supseteq \Omega(\mathcal{SU}_C(2, \OO{C})),
\end{equation} 
where
\begin{equation}
\begin{split}
\Sigma(\mathcal{SU}_C(2, \OO{C})) & \coloneqq \mathcal{SU}_C(2,\OO{C})^{\textnormal{sing}}= \{\xi \oplus \xi^{-1} \mid \xi \in \Pic^0(C) \}, \\
\Omega(\mathcal{SU}_C(2, \OO{C}))& \coloneqq  \Sigma(\mathcal{SU}_C(2, \OO{C}))^{\textnormal{sing}} =\{ \xi\oplus \xi \mid \xi \in \Pic^0(C)[2] \},
\end{split}
\end{equation}
where we consider bundles up to $S$-equivalence. 
\end{prop}

\subsection{Hyperelliptic action on the moduli space of rank two bundles} 

Assume $C$ is a smooth hyperelliptic curve of genus $g\geq 3$ and let $\sigma\in \Aut(C)$ be the hyperelliptic involution. We are interested in the following involution  
\begin{equation}\label{dualising+hyperelliptic}
\begin{split}
\sigma^\ast (-)^\ast \colon \mathcal{U}_C(2,0) & \to \mathcal{U}_C(2,0) \\
E & \mapsto \sigma^\ast E^\ast,
\end{split}
\end{equation}
given by the composition of the hyperelliptic involution and the dualising involution. Note that it preserves the closed subscheme $\mathcal{SU}_C(2, \OO{C})$ and the induced action is simply given by $\sigma^\ast$, as the dualising involution acts trivially on $\mathcal{SU}_C(2, \OO{C})$.

\begin{thm}[Hyperelliptic action on the moduli space of rank two bundles]\label{hyperelliptic-moduli-bundles} 
Let $C$ be a smooth hyperelliptic curve of genus $g\geq 3$. The fixed locus of \eqref{dualising+hyperelliptic} contains the singular locus $\Sigma(\mathcal{U}_C(2,0))$ and it is an irreducible subvariety of dimension $3g-1$. When $g=3$ it is a divisor, and the quotient of $\mathcal{U}_C(2, 0)$ by \eqref{dualising+hyperelliptic} is smooth. 

Similarly, when restricting to $\mathcal{SU}_C(2,\OO{C})$, we have that the fixed locus of $\sigma$ contains the singular locus $\Sigma(\mathcal{SU}_C(2,\OO{C}))$ and it is an irreducible subvariety of dimension $2g-1$. When $g=3$ it is a divisor, and the corresponding quotient of $\mathcal{SU}_C(2, \OO{C})$ is smooth. 
\end{thm}

\begin{proof}  According to Remark \ref{triviality-determinant-morphism}, the moduli space $\mathcal{U}_C(2,0)$ looks \'etale locally like a product $\mathcal{SU}_C(2,\OO{C})\times \Pic^0(C)$. The Galois cover 
\begin{equation}\label{Galois-cover-SU-U}
\mathcal{SU}_C(2,\OO{C})\times \Pic^0(C) \overset{\otimes}{\longrightarrow} \mathcal{U}_C(2,0)
\end{equation}
is $\Ztwo$-equivariant with respect to the action on $\mathcal{SU}_C(2,\OO{C})$ and $\mathcal{U}_C(2,0)$ defined by \eqref{dualising+hyperelliptic}. 

Moreover, it is strongly \'etale -- in the sense of \cite[Appendix D to Ch. 1]{GIT} -- so that it induces an \'etale map between the corresponding fixed loci and there is a cartesian diagram
\begin{equation} 
\begin{tikzcd}
\mathcal{SU}_C(2, \OO{C})\times \Pic^0(C) \ar[r, "\otimes"] \ar[d] &  \mathcal{U}_C(2, 0) \ar[d] \\
\left( \mathcal{SU}_C(2, \OO{C})/\Ztwo \right) \times \Pic^0(C) \ar[r] & \mathcal{U}_C(2, 0)/\Ztwo
\end{tikzcd}
\end{equation}
where the vertical arrows are the quotient morphisms and the bottom arrow is \'etale. This shows that one can deduce the result about $\mathcal{U}_C(2,0)$ from the corresponding result about $\mathcal{SU}_C(2, \OO{C})$ and the hyperelliptic involution.

According to Lemma \ref{action-on-Pic}, the singular locus $\Sigma(\mathcal{SU}_C(2,\OO{C}))$ is fixed by $\sigma$. 

Now if $E$ is a slope-stable rank two bundle such that $\sigma^\ast E \cong E$, we can always lift such an isomorphism to an equivariant structure $(E, \alpha)$. Now since $\det(E)\cong \OO{C}$ and we only have two equivariant structures on $\OO{C}$, we only have two possibilities: either $\wedge^2 \alpha = -1$ or $\wedge^2 \alpha=1$ (and in the latter case $\alpha_p=\pm \ide_{E_p}$ at each ramification point $p\in R$).

In the second case, by replacing $E$ with $E\otimes \OO{C}(D)$, where $D$ is the effective divisor of ramification points $p\in R$ where $\alpha_p=-\ide_{E_p}$, we can assume by Kempf's descent criterion \cite[Thm. 2.3]{Drezet} that $E$ is actually a pull-back from $\mathbb{P}^1$, contradicting stability.

In the first case, we can use Theorem \ref{equivariant-RR} to compute the local dimension of the fixed locus around $E$ in $\mathcal{SU}_C(2, \OO{C})$. Indeed, we have
\begin{equation}
\begin{split}
\dim T_{E}^{\Ztwo} \mathcal{SU}_C(2, \OO{C}) & = \dim H^1(C, \mathcal{E}nd_0(E))^{\Ztwo} \\
& = -\chi^{\Ztwo}( \mathcal{E}nd_0(E))= 2g-1.
\end{split}
\end{equation}
Irreducibility follows from \cite{Kumar}\footnote{To be precise, in \cite{Kumar}, the author focuses on the connected component of the fixed locus given by $\sigma$-invariant polystable bundles $E$ which admit an equivariant structure $\alpha: E\to \sigma^\ast E$ such that $\operatorname{tr}(\alpha_p)=0$ for all $p\in R$, but our argument in the previous proof shows that this is the only possible case for stable bundles.}, where it is proved for $\mathcal{SU}_C(2, \omega_C)$: we can use any theta-characteristic to get an isomorphism with $\mathcal{SU}_C(2, \OO{C})$ which is equivariant with respect to the hyperelliptic involution by Lemma \ref{action-on-Pic}. The last statement is \cite[Thm. 3]{Desale-Ramanan}.
\end{proof}

\begin{rmk} If one only considers the hyperelliptic involution $E\mapsto \sigma^\ast E$ on $\mathcal{U}_C(2,0)$, there is an induced non-trivial action on $\Pic^0(C)$ in \eqref{Galois-cover-SU-U}, so that the singular locus is not contained in the fixed locus anymore and the latter has dimension $2g-1$. Hence the corresponding quotient $\mathcal{U}_C(2,0)/\Ztwo$ is never smooth. 
\end{rmk}

\section{Moduli spaces of Higgs bundles and involutions}\label{section_Higgs}

\subsection{Moduli spaces of Higgs bundles} Let $C$ be a smooth curve. Recall a \emph{Higgs bundle} on $C$ is a pair $(E,\phi)$ where $E$ is a vector bundle on $C$ and $\phi\in \Hom_C(E, E\otimes \omega_C)$ is known as the \emph{Higgs field}. 

Slope-stability for Higgs bundles is defined in terms of $\phi$-invariant subbundles, and for $g\geq 2$ one gets a well-defined coarse moduli space $\mathcal{M}_C(r,d)$ parametrising $S$-equivalence classes of slope-semistable Higgs bundles on $C$ of  rank $r$ and degree $d$ \cite{Nitsure, Simpson-II}. It is a normal, irreducible, quasi-projective variety of dimension
\begin{equation}
\dim \mathcal{M}_C(r, d)= r^2(2g-2)+2,
\end{equation}  
which contains a non-empty smooth open subset $\mathcal{M}^s_C(r,d)\subseteq \mathcal{M}_C(r,d)$ parametrising isomorphism classes of slope-stable Higgs bundles. \\

When the bundle $E$ itself is slope-stable, then $(E,\phi)$ is slope-stable for every Higgs field $\phi\in \Hom_C(E,E\otimes \omega_C)$. Since $\Hom_C(E,E\otimes \omega_C)\cong \Ext^1_C(E, E)^\ast$ by Serre duality, and the latter can be identified with the fibre of $T^\ast \mathcal{U}^s_C(\rk(E), \deg(E))$ at $E$, the rational map
\begin{equation}
\mathcal{M}_C(r,d) \dashrightarrow \mathcal{U}_C(r,d)
\end{equation}
which forgets the Higgs field coincides with the cotangent bundle over the stable locus $\mathcal{U}^s_C(r,d)$. \\

Having fixed a line bundle $\xi\in \Pic^d(C)$, by taking the fibre over $(\xi, 0)$ of the determinant morphism
\begin{equation}\label{determinant-morphism-Higgs}
(\det, \text{tr})\colon \mathcal{M}_C(r,d)\to T^\ast \Pic^d(C),
\end{equation}
one gets the moduli space $\mathcal{M}_C(r,\xi)$ of Higgs bundles with fixed determinant, or $\SL(r,\mathbb{C})$-Higgs bundles \cite{Hitchin-stable-bundles, Kiem}. It is a normal, irreducible, quasi-projective variety of dimension
\begin{equation}
\dim \mathcal{M}_C(r,\xi)=(r^2-1)(2g-2),
\end{equation}
which contains a non-empty smooth open subset $\mathcal{M}^s_C(r,\xi)$ parametrising stable Higgs bundles. \\

Once again, by forgetting the Higgs field, one gets a rational map
\begin{equation}
\mathcal{M}_C(r,\xi) \dashrightarrow \mathcal{SU}_C(r,\xi)
\end{equation}
which agrees with the cotangent bundle over the stable locus  $\mathcal{SU}^s_C(r,\xi)$.

 \begin{rmk}[Local triviality of determinant morphism]\label{triviality-determinant-morphism-Higgs} The determinant morphism \eqref{determinant-morphism-Higgs} becomes trivial after a finite \'etale cover. Indeed, we have the following cartesian diagram
\begin{equation} 
\begin{tikzcd}[column sep=2cm]
\mathcal{M}_C(r, \xi)\times T^\ast \Pic^d(C) \ar[r] \ar[d] &  \mathcal{M}_C(r, d) \ar[d, "{(\det, \operatorname{tr})}"] \\
T^\ast \Pic^d(C) \ar[r] & T^\ast \Pic^d(C)
\end{tikzcd}
\end{equation}
where the top and bottom arrow are described, respectively, by
\begin{equation}
\left((E, \phi), (\eta, \psi) \right) \mapsto \left( E\otimes \eta\otimes \xi^{-1}, \phi + \frac{\psi}{r} \right)
\end{equation}
and
\begin{equation}
(\eta, \psi) \mapsto (\eta^r \otimes \xi^{1-r}, \psi).
\end{equation}
Both horizontal arrows are in fact Galois covers with Galois group $\Pic^0(C)[r]$. In particular, the moduli spaces $\mathcal{M}_C(r, \xi)$ and $\mathcal{M}_C(r,d)$ are stably isosingular. 
\end{rmk}

 \subsection{Higgs bundles as torsion sheaves on the cotangent bundle} 
Let $C$ be a smooth curve of genus $g\geq 2$ and let $
S\coloneqq \mathbf{Spec}_C \left( \Sym^\bullet \omega_C^{-1}\right) $
be the total space of its canonical bundle. It is a quasi-projective symplectic surface that admits a natural compactification given by the geometrically ruled surface $
 \overline{S} \coloneqq \mathbf{Proj}_C \left( \Sym^\bullet (\OO{C}\oplus \omega_C^{-1}) \right)$ 
 obtained by adding a divisor $D_\infty$ at infinity. Denoting by $\pi\colon \overline{S}\to C$ the projection, we have that $\Pic(\overline{S})$ is generated by the subgroup $\pi^\ast \Pic(C)$ and the relative hyperplane bundle $\OO{\overline{S}}(1)$. In particular, 
 \begin{equation}\label{Neron-Severi-overline-K}
\NS(\overline{S})= \mathbb{Z}[f] \oplus \mathbb{Z}[D_\infty],
\end{equation}
where $[f]$ denotes the class of any fibre of $\pi$, so that $[f]^2=0$, $[f]\cdot [D_\infty]=1$ and $[D_\infty]^2=2-2g$.  \\

Since $\pi_\ast \OO{\overline{S}}(1)\cong \OO{C}\oplus \omega_C^{-1}$, we have a canonical section $
\mu \in H^0(\overline{S}, \OO{\overline{S}}(1))
$
which cuts out $D_\infty$. On the other hand, $\pi_\ast (\pi^\ast \omega_C \otimes \OO{\overline{S}}(1))\cong \omega_C \oplus \OO{C}$, so that we also get a section
$
\lambda \in H^0(\overline{S}, \pi^\ast \omega_C \otimes \OO{\overline{S}}(1))
$
whose zero locus is the zero section of $\omega_C$, whose image we still denote by $C\subseteq \overline{S}$ and which is then numerically equivalent to $[D_\infty] + (2g-2) [f]$.
Finally, the adjunction formula shows that
\begin{equation}\label{canonical-Kbar}
c_1(\omega_{ \overline{S}}) = -2 [D_\infty] \in \NS(\overline{S}).
\end{equation}

We fix once and for all a polarisation $H\in \Amp(\overline{S})$ such that
\begin{equation}\label{polarisation-on-K-bar}
H = k(2g-2)[f] + [D_\infty] \in \NS(\overline{S}),
\end{equation} 
for $k \gg 0$. The so-called \emph{spectral correspondence} relates semistable Higgs bundles to pure one-dimensional sheaves on $S$, which are parametrised by an open subset -- defined by the condition of being supported in $S\subseteq \overline{S}$ -- of the moduli space $\mathcal{M}_{\overline{S},H}(0, r[C], d+r(1-g))$ of $H$-Gieseker semistable sheaves on $\overline{S}$:

\begin{thm}[Spectral correspondence, {\cite[Lemma 6.8]{Simpson-II}}, {\cite[Prop. 3.6]{BNR}}] \label{spectral-correspondence} There is an open immersion  
\begin{equation} \label{spectral-correspondence-morphism}
\mathcal{M}_C(r,d) \to \mathcal{M}_{\overline{S},H}(0, r [C], d+r(1-g)).
\end{equation}
Namely, each family of slope-semistable Higgs bundles on $C$ of rank $r$ and degree $d$ induces a family of $H$-Gieseker semistable pure one-dimensional sheaves on $\overline{S}$ with $c_1=r [C]$ and $\chi=\ch_2=d+r(1-g)$ whose support does not intersect $D_\infty$.  
\end{thm}

\subsection{Symplectic structure and crepant resolutions}  The open subset of $\mathcal{M}_{\overline{S},H}(\ch)$ defined by \eqref{spectral-correspondence-morphism} parametrises sheaves on the quasi-projective symplectic surface $S\subseteq \overline{S}$. In particular, it is naturally endowed with a symplectic structure on the stable locus, constructed by Mukai \cite{Mukai} by means of the canonical symplectic structure on $S$. 

This defines a symplectic structure on the stable locus $\mathcal{M}^s_C(r,d)$, which coincides with the canonical symplectic structure on $T^\ast \mathcal{U}^s_C(r,d)$, seen as an open subset of $\mathcal{M}^s_C(r,d)$ \cite[Thm. 1.1]{BBG}. 
 
Similarly, the moduli space $\mathcal{M}^s_C(r, \xi)$ of stable Higgs bundles with fixed determinant $\xi\in \Pic^d(C)$, is a holomorphic symplectic submanifold of $\mathcal{M}^s_C(r, d)$, and the induced symplectic structure extends the canonical symplectic structure on the open subset $T^\ast \mathcal{SU}^s_C(r,d)$.  \\

When $r$ and $d$ are coprime, every semistable Higgs bundle is stable, so that the moduli spaces $\mathcal{M}_C(r,d)$ and $\mathcal{M}_C(r, \xi)$ are holomorphic symplectic manifolds \cite{Hitchin-self-duality}.  On the other hand, strictly polystable sheaves give rise to singular points of the moduli space. We will focus, for instance, on the case $d=0$ (or equivalently $d$ is a multiple of $r$):

\begin{thm}[Symplectic resolutions of moduli spaces of Higgs bundles, {\cite{Tirelli, Kiem, Bellamy-Schedler-character}}]\label{singularities-GL-Higgs} Given a smooth curve $C$ of genus $g\geq 2$, an integer $r\geq 2$ and a line bundle $\xi\in \Pic^0(C)$, the moduli spaces $\mathcal{M}_C(r,0)$ and $\mathcal{M}_C(r,\xi)$ are holomorphic symplectic varieties. Moreover we have the following cases:
\begin{enumerate}[label={\upshape(\roman*)}]
\item When $(g,r)=(2,2)$, $\mathcal{M}_C(r,0)$ (resp. $\mathcal{M}_C(r,\xi)$) admits a 10-dimensional (resp. 6-dimensional) symplectic resolution obtained by blowing-up the reduced singular locus.
\item When $(g,r)\neq (2,2)$, the moduli spaces $\mathcal{M}_C(r,0)$ and $\mathcal{M}_C(r,\xi)$ do not admit any crepant resolution. 
\end{enumerate}
\end{thm}

\begin{table}
\centering
\begingroup
\small
\begin{tabular}{|>{\centering\arraybackslash}m{\dimexpr 0.10\textwidth-2\tabcolsep-1.3333\arrayrulewidth\relax}|>{\centering\arraybackslash}m{\dimexpr 0.42\textwidth-2\tabcolsep-1.3333\arrayrulewidth\relax}|>{\centering\arraybackslash}m{\dimexpr 0.42\textwidth-2\tabcolsep-1.3333\arrayrulewidth\relax}|}
\hline
\multicolumn{1}{|c|}{} & \multicolumn{1}{c|}{$g=2$} & \multicolumn{1}{c|}{$g\geq 3$} \\
\hline
$r=1$ & \multicolumn{2}{c|} {$\mathcal{M}_{C}(1, 0)\cong T^\ast \Pic^0(C)\cong \Pic^0(C)\times H^0(C, \omega_C)$ } \\
\hline
$r=2$ & $\mathcal{M}_{C}(2,0)$ and  $\mathcal{M}_{C}(2, \xi)$ symplectic varieties of dimension 10 and 6 with a crepant resolution  \cite{Tirelli, Kiem} & $\mathcal{M}_{C}(2, 0)$ and $\mathcal{M}_{C}(2, \xi)$ symplectic varieties with terminal singularities, no crepant resolution \cite{Tirelli, Kiem} \\
\hline
$r \geq 3$ & \multicolumn{2}{>{\centering\arraybackslash}p{0.82\textwidth}|} {$\mathcal{M}_{C}(r,0)$ and $\mathcal{M}_C(r, \xi)$ symplectic varieties with terminal singularities, no crepant resolution \cite{Tirelli, Kiem, Bellamy-Schedler-character}} \\
\hline
\end{tabular}
\endgroup
\caption{Moduli spaces $\mathcal{M}_{C}(r, 0)$ and $\mathcal{M}_C(r, \xi)$ of semistable Higgs bundles on a smooth curve $C$ of genus $g$, for $r\geq 1$ and $\xi \in \Pic^0(C)$.}
\label{table:GL}
\end{table}

\subsection{Singularities of moduli spaces of rank two Higgs bundles} We will from now on be focusing on the case of rank two and degree zero Higgs bundles, for which we can get an explicit description of the singular locus, analogous to the one given for moduli spaces of bundles in Proposition \ref{singularities-U-SU}, by simply looking at the Jordan--H\"older decomposition of each $S$-equivalence class:

\begin{prop}[Singularities of $\mathcal{M}_C(2,0)$ and $\mathcal{M}_C(2,\OO{C})$, {\cite[\textsection 2]{Kiem}}] \label{sing-Higgs} Let $C$ be a smooth curve of genus $g\geq 2$. The stratification of $\mathcal{M}_C(2,0)$ by symplectic leaves of Proposition \ref{stratification-symplectic-leaves} is induced by the filtration
\begin{equation}\label{stratification-GL}
\mathcal{M}_C(2,0) \supseteq \Sigma(\mathcal{M}_C(2,0)) \supseteq \Omega(\mathcal{M}_C(2,0)),
\end{equation}
where
\begin{equation}
\begin{split}
\Sigma(\mathcal{M}_C(2,0)) & \coloneqq \mathcal{M}_C(2,0)^{\textnormal{sing}}= \{(\xi, \phi) \oplus (\xi', \phi') \mid (\xi,\phi), (\xi', \phi') \in T^\ast \Pic^0(C)  \}, \\
\Omega(\mathcal{M}_C(2,0))& \coloneqq  \Sigma(\mathcal{M}_C(2,0))^{\textnormal{sing}} =\{  (\xi, \phi) \oplus (\xi, \phi) \mid (\xi, \phi) \in T^\ast \Pic^0(C)\},
\end{split}
\end{equation}
where we consider Higgs bundles up to $S$-equivalence. 

Similarly, for $\mathcal{M}_C(2, \OO{C})$ we have
\begin{equation}\label{stratification-SL}
\mathcal{M}_C(2, \OO{C}) \supseteq \Sigma(\mathcal{M}_C(2, \OO{C})) \supseteq \Omega(\mathcal{M}_C(2, \OO{C})),
\end{equation} 
where
\begin{equation}
\begin{split}
\Sigma(\mathcal{M}_C(2, \OO{C})) & \coloneqq \mathcal{SU}_C(2,\OO{C})^{\textnormal{sing}}= \{ (\xi, \phi) \oplus (\xi^{-1}, -\phi) \mid (\xi, \phi) \in T^\ast \Pic^0(C) \}, \\
\Omega(\mathcal{M}_C(2, \OO{C}))& \coloneqq  \Sigma(\mathcal{M}_C(2, \OO{C}))^{\textnormal{sing}} =\{  (\xi, 0) \oplus (\xi, 0) \mid \xi \in \Pic^0(C)[2] \},
\end{split}
\end{equation}
where we consider bundles up to $S$-equivalence. 
\end{prop}

Note that $\Sigma(\mathcal{M}_C(2,0))$ is isomorphic to $\Sym^2 (T^\ast \Pic^0(C))$, whose singular locus is a diagonally embedded copy of $T^\ast \Pic^0(C)$, so that
\begin{equation}
\begin{split}
\dim \Sigma(\mathcal{M}_C(2,0)) & = 4g, \\
\dim \Omega(\mathcal{M}_C(2,0)) & = 2g.
\end{split} 
\end{equation}

On the other hand, the subvariety $\Sigma(\mathcal{M}_C(2, \OO{C}))$ is isomorphic to the quotient of the cotangent bundle to the Jacobian $T^\ast \Pic^0(C)$ by the involution $(\xi, \phi) \mapsto (\xi^{-1}, -\phi)$, whose fixed locus is  given by the $2^{2g}$ points of $\Omega(\mathcal{M}_C(2, \OO{C}))$, corresponding to 2-torsion elements of the Jacobian.

\subsection{Involutions on moduli spaces of rank two Higgs bundles}  We will now define the involutions on $\mathcal{M}_C(2,0)$ and $\mathcal{M}_C(2, \OO{C})$ mentioned in Theorems \ref{main_thm_GL}, \ref{main_thm_SL} and \ref{reformulation}. \\

The first natural way of defining automorphisms of the moduli space is by lifting automorphisms of the base curve: if $\sigma\in \Aut(C)$ is an involution, then for a polystable pair $(E,\phi)$ in $\mathcal{M}_C(2,0)$, the assignment
\begin{equation}\label{lifting-involution-Higgs}
(E, \phi) \mapsto (\sigma^\ast E, \sigma^\ast \phi),
\end{equation}
where we are implicitly identifying $\sigma^\ast \omega_C\cong \omega_C$ by means of the natural equivariant structure on $\omega_C$, defines a symplectic involution on $\mathcal{M}_C(2,0)$. Indeed, the equivariant structure defines a lift of $\sigma$ to an involution of the ruled surface $\overline{S}$ -- which preserves the symplectic form on $S$ -- fixing $H$ and $C$, and hence naturally acts via pull-back on the Gieseker moduli space  $\mathcal{M}_{\overline{S},H}(0, 2[C], 2-2g)$ \cite[\textsection 3]{Franco}. \\

On the other hand, by considering the involution on $\overline{S}$ induced by multiplication by $-1$ on the fibres of $\omega_C$ -- which is anti-symplectic on $S$ -- and the induced action on the Gieseker moduli space, one gets the anti-symplectic involution
\begin{equation}\label{minus-one-explicit}
(E, \phi) \mapsto (E, -\phi).
\end{equation}

Note that the composition $(E, \phi)\mapsto (\sigma^\ast E, -\sigma^\ast \phi)$ of \eqref{lifting-involution-Higgs} and \eqref{minus-one-explicit} is simply induced by the twisted equivariant structure $\omega_C[-1]$ \cite[\textsection 3]{Franco}. \\

Finally, since $\mathcal{M}_C(2,0)$ can be regarded as a compactified Jacobian of the family of spectral curves, one would like a family version of the dualising involution -- and possibly twisting by a line bundle to fix the degree, as in \cite{Sacca-Prym}. This can be achieved by means of the spectral correspondence and the derived dual on the ruled surface $\overline{S}$, but one has to be careful with the choice of polarisation:

\begin{prop}[Duality for pure one-dimensional sheaves, \cite{Sacca-Prym}]\label{dualising-involution} Let $(X,H)$ be a smooth polarised surface and $\ch\in H^{2\ast}(X,\mathbb{Q})$ the Chern character of a one-dimensional $H$-Gieseker semistable sheaf. 

Given a line bundle $\xi \in \Pic(X)$, the assignment
\begin{equation}\label{dualisation-formula}
F \mapsto R\HHom_X(F, \OO{X})[1]\otimes \xi \cong \EExt^1_X(F, \xi)
\end{equation} 
defines an involution on the moduli space $\mathcal{M}_{X,H}(\ch)$ provided that the following two conditions on $H$ and $\xi$ are satisfied for every $H$-Gieseker semistable sheaf $F$ with $\ch(F)=\ch$:
\begin{enumerate}[label={\upshape(\roman*)}]
\item $2\chi(F)+c_1(F)\cdot (c_1(\omega_X)-c_1(\xi))=0$;
\item For all subsheaves $E\subseteq F$ one has
\begin{equation}\label{condition-stability-invariant}
\frac{c_1(E)\cdot (c_1(\omega_X)-c_1(\xi))}{c_1(E)\cdot H}=\frac{c_1(F)\cdot (c_1(\omega_X)-c_1(\xi))}{c_1(F)\cdot H}.
\end{equation}
\end{enumerate}
\end{prop}

\begin{proof} First note that for each pure one-dimensional sheaf, the complex  \begin{equation}
R\HHom_X(F, \OO{X})[1]\end{equation} is indeed a sheaf \cite[Prop. 1.1.10]{Huy-Lehn} and the assignment \eqref{dualisation-formula} is well-behaved in families by \cite[Lemma 3.8]{Sacca-Prym}. Hence the involution is well defined provided that $\ch$ is invariant and stability is preserved.

Since $c_1(F)$ is clearly preserved by \eqref{dualisation-formula} as the Fitting support is invariant, a direct computation shows that
\begin{equation}\label{hilb-poly-dual}
P_H(\EExt^1_X(F, \xi), t)=t(c_1(F)\cdot H)-\chi(F) - c_1(F)\cdot (c_1(\omega_X)-c_1(\xi)),
\end{equation}
so that invariance of $\ch$ is precisely condition $\mathrm{(i)}$. 

As for stability, there is a one-to-one correspondence between subsheaves of $F$ and quotient sheaves of $\EExt^1_X(F, \xi)$, given by associating to $E\subseteq F$ the induced quotient $\EExt^1_X(F, \xi)\to \EExt^1_X(E, \xi)$. Using \eqref{hilb-poly-dual}, we see that $\EExt^1_X(F, \xi)$ if $H$-Gieseker semistable if and only if
\begin{equation}
\frac{\chi(E)+c_1(E)\cdot (c_1(\omega_X)-c_1(\xi))}{c_1(E)\cdot H} \leq \frac{\chi(F)+c_1(F)\cdot (c_1(\omega_X)-c_1(\xi))}{c_1(F)\cdot H},
\end{equation}
for all subsheaves $E\subseteq F$. In other words, \eqref{dualisation-formula} turns $H$-Gieseker stability into $(c_1(\omega_X)-c_1(\xi))$-twisted $H$-Gieseker stability, and the two are equivalent under condition \eqref{condition-stability-invariant}. 
\end{proof}

In the case of compact symplectic surfaces considered in \cite{Sacca-Prym}, it is essentially impossible to find invariants such that the conditions of Proposition \ref{dualising-involution} are satisfied \cite[\textsection 3.5]{Sacca-Prym}, so one has to settle for a birational involution. A crucial step in our construction is to show that such an involution preserves stability for sheaves supported in the open subset $S\subseteq \overline{S}$:

\begin{cor}[Duality for Higgs bundles]
\label{dualising-bundle} For a semistable pair $(E,\phi)$ in $\mathcal{M}_C(2,0)$, the assignment
\begin{equation}\label{duality-explicit}
(E, \phi) \mapsto (E^\ast, \phi^t)
\end{equation}
defines an anti-symplectic involution on $\mathcal{M}_C(2,0)$.
\end{cor} 

\begin{proof} As above, we once again regard $\mathcal{M}_C(2,0)$ as an open subset of the Gieseker moduli space $\mathcal{M}_{\overline{S},H}(0, 2[C], 2-2g)$ and consider the involution defined in Proposition \ref{dualising-involution} with $\xi=\OO{\overline{S}}(-C)$, where $C$ is considered as the zero section in $S\subseteq \overline{S}$. Note that both conditions are satisfied if we restrict to $\mathcal{M}_C(2,0)$. Indeed, the support of any sheaf $F$ supported in $S\subseteq \overline{S}$ does not intersect $D_\infty$, so that by our choice \eqref{polarisation-on-K-bar} of $H$ and \eqref{canonical-Kbar} we have
\begin{equation}
\frac{c_1(F)\cdot (c_1(\omega_{\overline{S}})+[C])}{c_1(F)\cdot H}=\frac{(2g-2)c_1(F)\cdot [f]}{k(2g-2)c_1(F)\cdot [f]}=\frac{1}{k}. 
\end{equation}
Note that this ratio does not depend on $F$, so that condition \eqref{condition-stability-invariant} in Proposition \ref{dualising-involution} is satisfied for any sheaf in $\mathcal{M}_C(2,0)\subseteq \mathcal{M}_{\overline{S},H}(0, 2[C], 2-2g)$.  Finally, \eqref{duality-explicit} is anti-symplectic by \cite[Prop. 3.11]{Sacca-Prym}. 
\end{proof}

\subsection{$\mathcal{M}_C(2, \OO{C})$ as fixed locus of a symplectic involution} 

Recall that the moduli space $\mathcal{M}_C(2, \OO{C})$ is defined as a fibre of the determinant morphism
\begin{equation}
(\det,\operatorname{tr})\colon \mathcal{M}_C(2,0)\to T^\ast \Pic^0(C). 
\end{equation}

In the specific case of rank two Higgs bundles, we can make use of the following useful characterisation:

\begin{lem}\label{characterisation-SL} A polystable pair $(E,\phi)$ in $\mathcal{M}_C(2, 0)$ represents a point in the moduli space $\mathcal{M}_C(2, \OO{C})$ if and only if $\det E \cong \OO{C}$ and the induced skew-symmetric isomorphism $f\colon  E\cong E^\ast$ fits into the following commutative diagram
\begin{equation}\label{condition-SL}
\begin{tikzcd}
E \ar[r, "\phi"] \ar[d, "f"] & E\otimes \omega_C \ar[d, "f\otimes \ide_{\omega_C}"] \\
E^\ast \ar[r, "-\phi^t"] & E^\ast \otimes \omega_C 
\end{tikzcd}
\end{equation}
\end{lem}

\begin{proof} This simply follows from the exceptional isomorphism $\operatorname{Sp}(2,\mathbb{C})\cong \SL(2,\mathbb{C})$ -- and the associated Lie algebras. More explicitly, if $V$ is a 2-dimensional complex vector space endowed with a skew-symmetric non-degenerate pairing $f\colon V\cong V^\ast$, then an endomorphism $\phi\in \End(V)$ satisfies $\operatorname{tr}(\phi)=0$ if and only if it is skew-symmetric with respect to $f$, i.e. the diagram
\begin{equation}
\begin{tikzcd}
V \ar[r, "\phi"] \ar[d, "f"] & V \ar[d, "f"] \\
V^\ast \ar[r, "-\phi^t"] & V^\ast 
\end{tikzcd}
\end{equation}
is commutative.
\end{proof}

It follows immediately from Lemma \ref{characterisation-SL} that we can realise $\mathcal{M}_C(2, \OO{C})$ as a connected component of the fixed locus of the involution
\begin{equation}\label{involution-SL}
(E,\phi)\mapsto (E^\ast, -\phi^t)
\end{equation}
on $\mathcal{M}_C(2,0)$, which is essentially described by the commutativity of  \eqref{condition-SL}. Note that \eqref{involution-SL} is well-defined as the composition of \eqref{minus-one-explicit} and  \eqref{duality-explicit} and is in fact a symplectic involution. 

\begin{prop}[$\mathcal{M}_C(2,\OO{C})$ as a fixed locus]\label{SL-fixed-in-GL} For a Higgs pair $(E,\phi)$ in $\mathcal{M}_C(2,0)$, the assignment
\begin{equation}
\begin{split}
\tau\colon \mathcal{M}_C(2,0) & \to \mathcal{M}_C(2,0) \\
(E,\phi) & \mapsto (E^\ast, -\phi^t)
\end{split}
\end{equation}
defines a symplectic involution of $\mathcal{M}_C(2,0)$ whose fixed locus contains $\mathcal{M}_C(2, \OO{C})$ as a connected component.
\end{prop}

\begin{rmk}[Fixed locus of $\tau$] Note that the Galois cover 
\begin{equation}\label{Galois-cover-Higgs-I}
\mathcal{M}_C(2, \OO{C})\times T^\ast \Pic^0(C) \to \mathcal{M}_C(2, 0)
\end{equation}
from Remark \ref{triviality-determinant-morphism-Higgs} is $\Ztwo$-equivariant with respect to the action of $\tau$ on the target and the action
\begin{equation}\label{duality-rank-one}
(\xi, \phi)\mapsto (\xi^{-1}, -\phi)
\end{equation}
on $T^\ast \Pic^0(C)$. In fact, the assignment \eqref{duality-rank-one} defines a $\Ztwo$-action on $\mathcal{M}_C(2, \OO{C})\times T^\ast \Pic^0(C)$ which commutes with the Galois action and hence descends to $\mathcal{M}_C(2, 0)$ as $\tau$. 
The fixed locus of \eqref{duality-rank-one} on $\mathcal{M}_C(2, \OO{C})\times T^\ast \Pic^0(C)$ is given by several disjoint copies of the fibre $\mathcal{M}_C(2, \OO{C})$, indexed by torsion line bundles in $\Pic^0(C)[2]$, which all map to $\mathcal{M}_C(2, \OO{C})\subseteq \mathcal{M}_C(2,0)$ under \eqref{Galois-cover-Higgs-I}. 

The induced map between fixed loci is not surjective though -- in fact \eqref{Galois-cover-Higgs-I} is not injective on $\Ztwo$-orbits, and hence is not strongly \'etale \cite[Appendix D to Ch. 1]{GIT} -- and hence the fixed locus of $\tau$ is strictly bigger than $\mathcal{M}_C(2, \OO{C})$: for instance, the semistable pairs $(\xi\oplus \OO{C}, 0)$, for $\xi \in \Pic^0(C)[2]$, are all fixed by $\tau$. 
\end{rmk}

\section{Proof of the main results}\label{section_final_proofs}

The proofs of Theorems \ref{main_thm_GL}, \ref{main_thm_SL} and \ref{reformulation} are articulated as follows:
\begin{itemize}
\item Prove Theorem \ref{main_thm_SL} by the following steps:
\begin{itemize}
\item Relate the structure of the moduli space $\mathcal{M}_C(2, \OO{C})$ around the singular point $(\OO{C}^{\oplus 2}, 0)$ in the deepest singular stratum $\Omega(\mathcal{M}_C(2, \OO{C}))$ (see Proposition \ref{sing-Higgs}) to that of the adjoint orbit $\mathcal{N}_{\leq 3}(\mathbb{C}^6)$ studied in \textsection \ref{section_local_model} by means of a two-to-one ramified cover on analytic germs (Theorem \ref{local-structure-SL})
\begin{equation}\label{crucial-double-cover}
\mu \colon (\mathcal{M}_C(2,\OO{C}), (\OO{C}^{\oplus 2}, 0)) \to (\mathcal{N}_{\leq 3}(\mathbb{C}^6), 0).
\end{equation}
\item Assuming $C$ is hyperelliptic, show that the covering involution of \eqref{crucial-double-cover} extends to a global involution of $\mathcal{M}_C(2, \OO{C})$, induced by the hyperelliptic involution (Remark \ref{covering-involution-mu} and Lemma \ref{tangent-action-sigma-tau}).
\item The quotient $\pi \colon \mathcal{M}_C(2, \OO{C})\to Y$ by the action of the hyperelliptic involution admits a stratification by symplectic leaves, and the locus of points in $Y$ where the blow-up of the reduced singular locus of $Y$ is not a crepant resolution is a union of strata (Corollary \ref{locus-non-crepant-resolution-clopen}). 
\item By exploiting the previous analysis around the point $(\OO{C}^{\oplus 2}, 0)$ in the deepest stratum, we show that each connected component of the symplectic leaves of $Y$ contains a point at which the blow-up is a crepant resolution, which concludes the proof. 
\end{itemize}
\item Reduce Theorem \ref{main_thm_GL} to Theorem \ref{main_thm_SL} by a suitable equivariant version of Remark \ref{triviality-determinant-morphism-Higgs}, as in the proof of Theorem \ref{hyperelliptic-moduli-bundles}. 
\item Theorem \ref{reformulation} follows immediately from Theorems \ref{main_thm_GL} and \ref{main_thm_SL} since $\mathcal{M}_C(2, \OO{C})$ is a connected component of the fixed locus of \eqref{involution-SL} on $\mathcal{M}_C(2,0)$ (Proposition \ref{SL-fixed-in-GL}).
\end{itemize}

The key technical result we need for the first step is given by Kaledin and Lehn's analysis of the singularities of moduli spaces of sheaves on symplectic surfaces, which also applies to Higgs bundles through the spectral correspondence. In order to state the result, we need to set up some notation about Hamiltonian reductions of symplectic vector space -- see \cite[Ch. 1]{Ginzburg}, or specifically \cite[\textsection 3.3]{KL}. 

\subsection{Hamiltonian reduction of symplectic vector spaces} Let $(V,\omega)$ be a complex symplectic vector space and $G$ a connected reductive subgroup of $\SP(V)$. The action of $G$ on $V$ is Hamiltonian and the moment map is the quadratic map given by
\begin{equation}\label{moment-map-V}
\begin{split}
Q\colon V & \to \mathfrak{g}^\ast, \\
v& \mapsto \left[ A \mapsto \frac{1}{2}\omega(Av, v) \right].
\end{split}
\end{equation}
We define the \textit{Hamiltonian reduction} of $V$ by $G$ to be the affine quotient 
\begin{equation}
V \sslash G \coloneqq Q^{-1}(0) / G,
\end{equation}
which is naturally endowed with a symplectic form on its smooth locus. 

We will be interested specifically in the following example.

\begin{ex} Assume $(V,\omega)$ is a complex symplectic vector space and $G$ is a complex reductive semisimple Lie group (we will only need the case $G=\SL(2,\mathbb{C})$). Denote by $\mathfrak{g}$ its Lie algebra, endowed with the non-degenerate Killing form $h\in \mathfrak{g}^\ast \otimes \mathfrak{g}^\ast$. In what follows, we will tacitly identify $\splie(V)$ with $\splie(V)^\ast$ by means of $\omega$, and $\mathfrak{g}$ with $\mathfrak{g}^\ast$ by means of $h$. 

The vector space $V\otimes \mathfrak{g}$ is naturally endowed with a symplectic form given by the product of $\omega$ and $h$:
\begin{equation}
\left(\sum_i v_i\otimes a_i, \sum_j w_j \otimes b_j \right) \mapsto \sum_{i,j}\omega(v_i, w_j)h(a_i, b_j). 
\end{equation}

There are two natural commuting Hamiltonian actions on $V\otimes \mathfrak{g}$, given by the adjoint action of $G$ on the right factor and the natural action of $\SP(V)$ on the left factor.

The moment map for the $G$-action
\begin{equation}\label{moment-map-G}
\begin{split}
Q_G\colon V\otimes \mathfrak{g} & \to \mathfrak{g}, \\
\sum_i v_i \otimes a_i & \mapsto \frac{1}{2}\sum_{i,j}\omega(v_i, v_j)[a_i, a_j]= \sum_{i<j}\omega(v_i,v_j) [a_i, a_j]
\end{split}
\end{equation}
is quadratic, and is given by the product of $\omega$ and the Lie bracket of $\mathfrak{g}$. 

On the other hand, the moment map for the $SP(V)$-action
\begin{equation}\label{moment-map-SP}
\begin{split}
Q_{\SP(V)}\colon V\otimes \mathfrak{g} & \to \splie(V), \\
\sum_i v_i \otimes a_i & \mapsto \frac{1}{2}\sum_{i,j}h(a_i, a_j) \omega(v_i, -)v_j
\end{split}
\end{equation}
is also quadratic, and is given by the product of the Killing form and the quadratic moment map \eqref{moment-map-V} for the $\SP(V)$-action on $V$. 
\end{ex}

\subsection{Local structure of moduli spaces of sheaves on symplectic surfaces}   Assume $(X,H)$ is a polarised symplectic surface and $E$ is a non-rigid $H$-Gieseker stable sheaf on $X$. By \cite{Mukai}, deformations of $E$ are unobstructed, so that the deformation space $\operatorname{Def}(E)$ is smooth and its tangent space $V\coloneqq \Ext^1_X(E,E)$ is a non-trivial symplectic vector space. Assume that the differential graded Lie algebra $R\Hom_X^\bullet (E,E)$ is formal -- this is satisfied if $X$ is a K3 surface by \cite{formality-conj} or if $E$ is the spectral sheaf associated to a Higgs bundle by \cite{Simpson-II}, which is our case of interest. 

Consider the polystable sheaf $F\coloneqq E\oplus E$ as a point of the moduli space $\mathcal{M}_{X,H}(\ch(F))$ of $H$-Gieseker semistable sheaves with Chern character $\ch(F)$. Note that there is a closed immersion
\begin{equation}\label{singular-locus-M2}
\begin{split}
\operatorname{Sym}^2 \mathcal{M}_{X,H}(\ch(E)) & \to \mathcal{M}_{X,H}(\ch(F)), \\
(E_1, E_2) & \mapsto E_1 \oplus E_2,
\end{split}
\end{equation} 
and the point corresponding to $F$ lies in the image of the diagonal $\mathcal{M}_{X,H}(\ch(E))\subseteq \operatorname{Sym}^2 \mathcal{M}_{X,H}(\ch(E))$.

When $\ch(E)$ is primitive and $H$ is generic,  \eqref{singular-locus-M2} describes exactly the singular locus of $\mathcal{M}_{X,H}(\ch(F))$, at least when $\dim V \geq 4$. In general, the singular locus might have several irreducible components, corresponding to the possible Jordan--H\"older decompositions of polystable sheaves in $\mathcal{M}_{X,H}(\ch(F))$.  

The following crucial result gives an explicit description of an analytic neighbourhood of the moduli space $\mathcal{M}_{X,H}(\ch(F))$ around the point $F$:

\begin{prop}[Local structure of $\mathcal{M}_{X,H}(\text{ch}(F))$, {\cite[Prop. 2.2]{KL}}]\label{Kaledin-Lehn-crucial} 
There is an isomorphism of analytic germs of symplectic varieties
\begin{equation}\label{isomorphism-of-germs}
(\mathcal{M}_{X,H}(\ch(F)), F) \cong (V \otimes  \gl(2, \mathbb{C}) \sslash \GL(2,\mathbb{C}), 0)
\end{equation}
where the right-hand side is the Hamiltonian reduction with respect to the adjoint action on the right factor of $V \otimes  \gl(2, \mathbb{C})$.  

Moreover there is a map of analytic germs of symplectic varieties
\begin{equation} \label{moment-map-Sp}
\mu \colon (\mathcal{M}_{X,H}(\ch(F)), F) \to (V \times \mathcal{N}_{\leq 3}(V), (0,0))
\end{equation}
with the following properties:
\begin{enumerate}[label={\upshape(\roman*)}]
\item $\mu$ is a two-to-one \'etale cover over the locus $V \times \mathcal{N}_{3}(V)$.
\item $\mu$ is one-to-one over the locus $V \times \mathcal{N}_{\leq 2}(V)$.
\item $\mu$ identifies the subvariety $\operatorname{Sym}^2 \mathcal{M}_{X,H}(\ch(E))$ given by \eqref{singular-locus-M2} with $V \times \mathcal{N}_{\leq 1}(V)$ and its diagonal with $V \times \{0\}$.
\end{enumerate}
\end{prop}

Assuming the existence of \eqref{isomorphism-of-germs}, we recall how to construct \eqref{moment-map-Sp}, as it will be useful later on. By means of the trace map $\operatorname{tr}\colon \gl(2, \mathbb{C})\to \mathbb{C}$ one decomposes
\begin{equation}\label{decomposition-trace-map}
V \otimes  \gl(2, \mathbb{C}) \sslash \GL(2,\mathbb{C}) \cong V \times (V \otimes  \mathfrak{sl}(2, \mathbb{C}) \sslash \operatorname{SL}(2,\mathbb{C})),
\end{equation}
so that we will define $\mu\colon V \otimes  \mathfrak{sl}(2, \mathbb{C}) \sslash \operatorname{SL}(2,\mathbb{C})\to \mathcal{N}_{\leq 3}(V)$ and then get \eqref{moment-map-Sp} by simply multiplying by $\ide_V$.

In order to do so, we consider the Hamiltonian $\SP(V)$-action on the left factor of $V \otimes  \mathfrak{sl}(2, \mathbb{C})$, which commutes with the $\operatorname{SL}(2,\mathbb{C})$-action. The map we are looking for is precisely the corresponding moment map \eqref{moment-map-SP}:
\begin{equation}
\mu\colon V \otimes  \mathfrak{sl}(2, \mathbb{C}) \to \mathfrak{sp}(V),
\end{equation}
given by the product of the Killing form on $\mathfrak{sl}(2, \mathbb{C})$ and the quadratic $\SP(V)$-moment map on $V$, which descends to the Hamiltonian reduction as the two actions commute. 

\begin{rmk}[Residual action of $\SP(V)$ on $V \otimes  \mathfrak{sl}(2, \mathbb{C}) \sslash \operatorname{SL}(2,\mathbb{C})$]\label{covering-involution-mu} Since the actions of $\SP(V)$ and $\operatorname{SL}(2,\mathbb{C})$ on $V\otimes \mathfrak{sl}(2, \mathbb{C})$ commute, the former descends to the Hamiltonian reduction $V \otimes  \mathfrak{sl}(2, \mathbb{C}) \sslash \operatorname{SL}(2,\mathbb{C})$, and the moment map
\begin{equation}
\mu\colon V \otimes  \mathfrak{sl}(2, \mathbb{C}) \sslash \operatorname{SL}(2,\mathbb{C}) \to \mathcal{N}_{\leq 3}(V)
\end{equation}
is $\SP(V)$-equivariant. The involution $-\ide_V\in \SP(V)$ acts trivially on the target $\mathcal{N}_{\leq 3}(V)$, but non-trivially on the Hamiltonian reduction $V \otimes  \mathfrak{sl}(2, \mathbb{C}) \sslash \operatorname{SL}(2,\mathbb{C})$, and is indeed the covering involution associated with the map $\mu$ \cite[\textsection 4]{KL}. 
\end{rmk}

\subsection{Local structure of $\mathcal{M}_C(2,0)$ and $\mathcal{M}_C(2, \OO{C})$} 
Assume $C$ is a smooth curve of  genus $g\geq 2$.  By Proposition \ref{stratification-symplectic-leaves}, the local structure of the moduli spaces $\mathcal{M}_C(2,0)$ and $\mathcal{M}_C(2, \OO{C})$ of Higgs bundles around a point only depends on the connected component of the open stratum in Proposition \ref{sing-Higgs} the point belongs to.  We will hence focus on the point corresponding to the pair
\begin{equation}
(\OO{C}^{\oplus 2}, 0)=(\OO{C}, 0)\oplus (\OO{C}, 0),
\end{equation}
which lies in the deepest stratum $\Omega(\mathcal{M}_C(2,\OO{C}))\subseteq \Omega(\mathcal{M}_C(2,0))$ of the moduli spaces. 

\begin{rmk}[Action of ${\Pic^0(C)[2]}$ on $\mathcal{M}_C(2,\OO{C})$] \label{action-of-2-torsion} The torsion subgroup $\Pic^0(C)[2]$ of $\Pic^0(C)$ acts on $\mathcal{M}_C(2,\OO{C})$ by tensor product, and the action is transitive on the deepest stratum $\Omega(\mathcal{M}_C(2,\OO{C}))$. In particular any two points in $\Omega(\mathcal{M}_C(2,\OO{C}))$ have isomorphic analytic neighbourhoods in $\mathcal{M}_C(2,\OO{C})$.
\end{rmk}

We will make use of Theorem \ref{spectral-correspondence} and look at Higgs bundles as pure one-dimensional sheaves on the quasi-projective symplectic surface $S$, so that we can apply Proposition \ref{Kaledin-Lehn-crucial} to get a local description of the singularities of the moduli spaces. Note that the projective surface $\overline{S}$ is not symplectic, but since we only consider sheaves supported on the open subset $S$ and Mukai's construction of the symplectic structure is purely local, Proposition \ref{Kaledin-Lehn-crucial} still applies. In other words, the first order deformation space of the only stable summand $(\OO{C}, 0)$
\begin{equation}\label{splitting-tangent-space-rk1}
V\coloneqq \Ext^1_{\overline{S}}(\OO{C}, \OO{C})\cong H^1(C, \OO{C})\oplus H^0(C, \omega_C)
\end{equation}
is still a $2g$-dimensional symplectic vector space by Serre duality. Note that $V$ is the tangent space at $(\OO{C}, 0)$ of the moduli space $\mathcal{M}_C(1,0)$ and the direct sum decomposition \eqref{splitting-tangent-space-rk1} is induced by the fibre product decomposition $\mathcal{M}_C(1,0)\cong \Pic^0(C)\times H^0(C, \omega_C)$.

\begin{prop}[Local structure of $\mathcal{M}_C(2,0)$]\label{local-structure-GL} There is an isomorphism of analytic germs of symplectic varieties
\begin{equation}\label{isomorphism-of-germs-GL}
(\mathcal{M}_C(2,0), (\OO{C}^{\oplus 2}, 0)) \cong (V\times (V \otimes  \slie(2, \mathbb{C}) \sslash \SL(2,\mathbb{C})), (0,0) )
\end{equation}
where the right-hand side is the Hamiltonian reduction with respect to the adjoint action on the right factor of $V \otimes  \slie(2, \mathbb{C})$. 

Moreover there is a map of analytic germs of symplectic varieties
\begin{equation} 
\mu \colon (\mathcal{M}_C(2,0), (\OO{C}^{\oplus 2}, 0)) \to (V \times \mathcal{N}_{\leq 3}(V), (0, 0))
\end{equation}
with the following properties:
\begin{enumerate}[label={\upshape(\roman*)}]
\item $\mu$ is a two-to-one \'etale cover over the locus $V \times \mathcal{N}_{3}(V)$.
\item $\mu$ is one-to-one over the locus $V \times \mathcal{N}_{\leq 2}(V)$.
\item $\mu$ identifies the closed stratum $\Sigma(\mathcal{M}_C(2,0))$ with $V \times \mathcal{N}_{\leq 1}(V)$ and $\Omega(\mathcal{M}_C(2,0))$ with $V \times \{0\}$.
\end{enumerate}
\end{prop}

By restricting to the trace-free part, we get an analogous description for the moduli space $\mathcal{M}_{C}(2,\OO{C})$:

\begin{prop}[Local structure of $\mathcal{M}_C(2,\OO{C})$, {\cite[Lemma 3.3]{Kiem}}]\label{local-structure-SL} 
With the notation of Proposition \ref{local-structure-GL}, there is an isomorphism of analytic germs of symplectic varieties
\begin{equation}\label{isomorphism-of-germs-SL}
(\mathcal{M}_C(2,\OO{C}), (\OO{C}^{\oplus 2}, 0)) \cong (V \otimes  \slie(2, \mathbb{C}) \sslash \SL(2,\mathbb{C}), 0),
\end{equation}
and a map of analytic germs of symplectic varieties
\begin{equation} 
\mu \colon (\mathcal{M}_C(2,\OO{C}), (\OO{C}^{\oplus 2}, 0)) \to (\mathcal{N}_{\leq 3}(V), 0)
\end{equation}
with the following properties:
\begin{enumerate}[label={\upshape(\roman*)}]
\item $\mu$ is a two-to-one \'etale cover over the locus $V \times \mathcal{N}_{3}(V)$.
\item $\mu$ is one-to-one over the locus $V \times \mathcal{N}_{\leq 2}(V)$.
\item $\mu$ identifies the closed stratum $\Sigma(\mathcal{M}_C(2,0))$ with $\mathcal{N}_{\leq 1}(V)$.
\end{enumerate}
\end{prop}

\begin{ex}[Symplectic resolution for genus two] When $g=2$, we see that locally around $(\OO{C}^{\oplus 2}, 0)$ the moduli spaces $\mathcal{M}_C(2,0)$ and $\mathcal{M}_C(2,\OO{C})$ look like $V\times \mathcal{N}_{\leq 2}(V)$ and $\mathcal{N}_{\leq 2}(V)$ respectively, with $\dim V=4$ (there is no ramification in this case!). In particular, Theorem \ref{blow-up-nilpotent-cone} implies that blowing-up the reduced singular locus gives a crepant resolution (compare with Theorem \ref{singularities-GL-Higgs}). 
\end{ex}

\subsection{Hyperelliptic action on moduli spaces of Higgs bundles}

Assume $C$ is a smooth hyperelliptic curve of genus $g\geq 3$ and let $\sigma\in \Aut(C)$ be the hyperelliptic involution. We are interested in the following symplectic involution  
\begin{equation}\label{dualising+hyperelliptic-Higgs}
\begin{split}
\sigma\circ \tau\colon \mathcal{M}_C(2,0) & \to \mathcal{M}_C(2,0) \\
(E, \phi) & \mapsto (\sigma^\ast E^\ast, -\sigma^\ast \phi^t),
\end{split}
\end{equation}
given by the composition of the two commuting symplectic involutions \eqref{lifting-involution-Higgs} -- which we still denote by $\sigma$ -- and $\tau$ from Proposition \ref{SL-fixed-in-GL}. Note that it preserves the closed subscheme $\mathcal{M}_C(2, \OO{C})$ and the induced action is simply given by $\sigma$, as $\tau$ acts trivially on $\mathcal{M}_C(2, \OO{C})$.

\begin{lem}[Infinitesimal action of $\sigma\circ \tau$]\label{tangent-action-sigma-tau} The symplectic involution \eqref{dualising+hyperelliptic-Higgs} fixes the point $(\OO{C}^{\oplus 2}, 0)$. Via the isomorphism \eqref{isomorphism-of-germs-GL} it acts trivially on the first factor and as $-\ide_V$ on the second factor. 
\end{lem}

\begin{proof} First note that the actions of $\sigma$ and $\tau$ agree on the moduli space $\mathcal{M}_C(1, 0)$ of rank one Higgs bundles, by Lemma \ref{action-on-Pic}. In particular, one has $\sigma\circ \tau=\ide$ on $\mathcal{M}_C(1, 0)$ and the induced action on $V$ is trivial. 

In particular, $\sigma\circ \tau$ fixes the point $(\OO{C}^{\oplus 2}, 0)$. For the induced action on $V \otimes  \slie(2, \mathbb{C}) \sslash \SL(2,\mathbb{C})$ via \eqref{isomorphism-of-germs-GL}, note that $\sigma$ acts trivially on $\slie(2, \mathbb{C})$ while $\tau$ acts as transposition on $\slie(2, \mathbb{C})$, so that $\sigma \circ \tau$ acts as $\ide_V \otimes (-)^t$ on $V \otimes  \slie(2, \mathbb{C})$. We notice that every $a\in \slie(2, \mathbb{C})$ is $\SL(2,\mathbb{C})$-conjugate to $-a^t$. As a consequence, $\sigma\circ\tau$ acts as $-\ide$ on the quotient $V \otimes  \slie(2, \mathbb{C}) \sslash \SL(2,\mathbb{C})$, and the claim follows. 
\end{proof}

It follows from Lemma \ref{tangent-action-sigma-tau} that the involution $\sigma$, acting on $\mathcal{M}_C(2, \OO{C})$ -- which is fixed by $\tau$ --, coincides with the covering involution of $\mu$ according to Remark \ref{covering-involution-mu} in an analytic neighbourhood of $(\OO{C}^{\oplus 2}, 0)$. In particular, locally around $(\OO{C}^{\oplus 2}, 0)$ the quotient of $\mathcal{M}_C(2, \OO{C})$ by the $\Ztwo$-action generated by $\sigma$ admits a crepant resolution given by blowing-up the reduced singular locus $\Sigma(\mathcal{M}_C(2, \OO{C}))$ by Theorem \ref{blow-up-nilpotent-cone}. \\

Having taken care of points in the singular locus, one has to further investigate the local behaviour of $\sigma \circ \tau$ around fixed smooth points of $\mathcal{M}_C(2,0)$, in order to get an explicit description of the singular locus of the quotient. It will be enough to focus on $\mathcal{M}_C(2, \OO{C})$:

\begin{lem}[Fixed locus of $\sigma$ in $\mathcal{M}_C(2, \OO{C})$]\label{fixed-locus-of-sigma} The fixed locus $\operatorname{Fix}(\sigma) \subseteq \mathcal{M}_C(2,\OO{C})$ contains the singular locus $\Sigma(\mathcal{M}_C(2,\OO{C}))$ and is pure of dimension $4g-2$. 
\end{lem}
\begin{proof} The singular locus $\Sigma(\mathcal{M}_C(2, \OO{C}))$ is fixed by $\sigma$ as for any degree zero line bundle $\xi\in \Pic^0(C)$ one has $\sigma^\ast \xi \cong \xi^{-1}$ by Lemma \ref{action-on-Pic} and the induced action on global sections $H^0(C, \omega)$ is given by $-\ide$ by our choice of the equivariant structure on $\omega_C$. Moreover, it follows from Lemma \ref{tangent-action-sigma-tau} that the map $\mu$ of Proposition \ref{local-structure-SL} identifies the fixed locus of $\sigma$ around points of $\Omega(\mathcal{M}_C(2, \OO{C}))$ with $\mathcal{N}_{\leq 2}(V)$, so that $\Sigma(\mathcal{M}_C(2, \OO{C}))$ is contained in an irreducible component of the fixed locus of dimension
\begin{equation}
\dim \mathcal{N}_{\leq 2}(V) = 4g-2. 
\end{equation}
Now assume $(E,\phi)$ is a stable pair representing a smooth point in $\mathcal{M}_C(2, \OO{C})$ which is fixed by $\sigma$, i.e. there exists a commutative diagram
\begin{equation}\label{fixed-locus-higgs}
\begin{tikzcd}
E \ar[r, "\phi"] \ar[d, "\alpha"] & E\otimes \omega_C \ar[d, "\alpha"]  \\
\sigma^\ast E \ar[r, "\sigma^\ast \phi"] & \sigma^\ast E \otimes \omega_C
\end{tikzcd}
\end{equation}
Note that in the right-hand side vertical arrow we are implicitly tensoring with our choice of an equivariant structure on $\omega_C$. By stability of the pair we can assume $\alpha$ defines a $\Ztwo$-equivariant structure on $E$, so that commutativity of \eqref{fixed-locus-higgs} is equivalent to $\phi$ being a morphism of $\Ztwo$-equivariant bundles. 

By \cite[\textsection 7]{Nitsure}, the Zariski tangent space $T_{(E,\phi)}\mathcal{M}^s_C(2, \OO{C})$ at the stable pair $(E, \phi)$ fits into the long exact sequence
\begin{equation}\label{LES-defo}
\begin{tikzcd}[column sep=0.32in, row sep=small,
every label/.append style={font=\tiny}]
0 \ar[r] & \End_{C,0}(E) \arrow[d, phantom, ""{coordinate, name=Z}] \ar[r, "{[} -{,} \phi {]}"] & \Hom_{C,0}(E, E\otimes \omega_C) \ar[r] & T_{(E,\phi)}\mathcal{M}^s_C(2, \OO{C}) \ar[dll, rounded corners,
to path={ -- ([xshift=2ex]\tikztostart.east)
|- (Z) [near end]\tikztonodes
-| ([xshift=-2ex]\tikztotarget.west)
-- (\tikztotarget)}] \\
& \Ext^1_{C,0}(E,E)  \ar[r, "{[} -{,} \phi {]}"] & \Ext^1_{C,0}(E, E\otimes \omega_C) \ar[r] & 0.
\end{tikzcd}
\end{equation}

By taking $\Ztwo$-invariants we get the dimension of the fixed part as 
\begin{equation}
\dim T_{(E,\phi)}^{\Ztwo} \mathcal{M}^s_C(2, \OO{C})= -2\chi^{\Ztwo}(\EEnd_{C,0}(E)),
\end{equation} 
where $\EEnd_{C,0}(E)$ is endowed with the equivariant structure induced by $\alpha$ (see Theorem \ref{equivariant-RR} for the notation). 

Now we proceed as in the proof of Theorem \ref{hyperelliptic-moduli-bundles}. Since $\det(E)\cong \OO{C}$ we have two possibilities: either $\wedge^2 \alpha = -1$ or $\wedge^2 \alpha=1$ and hence $\alpha_p=\pm \ide_{E_p}$ at each ramification point $p\in R$ in the ramification locus.

In the second case, by replacing $E$ with $E\otimes \OO{C}(D)$, where $D$ is the effective divisor of ramification points where $\alpha_p=-\ide_{E_p}$, we can assume $E$ is actually a pull-back from $\mathbb{P}^1$. Moreover, commutativity of \eqref{fixed-locus-higgs} at ramification points -- given or choice of equivariant structure on the canonical bundle -- shows that $\phi$ actually factors through 
\begin{equation}
E\otimes \omega_C(-R) \cong E\otimes f^\ast \omega_{\mathbb{P}^1}\subseteq E\otimes \omega_C,
\end{equation}
so that $(E,\phi)$ is actually a pull-back from $\mathbb{P}^1$ as a pair, contradicting stability. 

In the first case, we can finally compute the dimension of the fixed locus around the fixed point $(E, \phi)$ by means of Theorem \ref{equivariant-RR}
\begin{equation}
\dim T_{(E,\phi)}^{\Ztwo} \mathcal{M}^s_C(2, \OO{C})= -2\chi^{\Ztwo}(\EEnd_{C,0}(E)) = 4g-2. 
\end{equation}
\end{proof}

\begin{rmk}[Irreducibility of the fixed locus of $\sigma$] We expect that $\operatorname{Fix}(\sigma)$ is actually irreducible, as it is suggested by the analogy with Theorem \ref{hyperelliptic-moduli-bundles}, which in particular proves irreducibility of the fixed locus when intersected with the locus of pairs with underlying stable bundle $T^\ast \mathcal{SU}^s_C(2, \OO{C})\subseteq \mathcal{M}_C(2,\OO{C})$.

Since it is smooth away from the singular locus by a standard argument due to Cartan \cite[\textsection 4]{Cartan}, and there is one irreducible component containing the singular locus $\Sigma(\mathcal{M}_C(2,\OO{C}))$ by Lemma \ref{fixed-locus-of-sigma}, this is equivalent to proving that the fixed locus is connected in $\mathcal{M}^s_C(2,\OO{C})$. 
\end{rmk}

\subsection{Local structure of the quotient and existence of a crepant resolution}

We will start by proving Theorem \ref{main_thm_SL}. Assume that $C$ is a smooth genus three hyperelliptic curve and consider the moduli space $\mathcal{M}_C(2, \OO{C})$, which is a 12-dimensional holomorphic symplectic variety which does not admit any crepant resolution by Theorem \ref{singularities-GL-Higgs}.  The hyperelliptic involution $\sigma\in \Aut(C)$ defines via \eqref{lifting-involution-Higgs} a symplectic involution on $\mathcal{M}_C(2,\OO{C})$, whose fixed locus is pure of dimension 10 and contains the irreducible singular locus $\Sigma(\mathcal{M}_C(2,\OO{C}))$ by Lemma \ref{fixed-locus-of-sigma}. 

\begin{lem}[Singularities of the quotient of $\mathcal{M}_C(2, \OO{C})$] \label{singularities-quotient} Let $\pi: \mathcal{M}_C(2,\OO{C}) \to Y$ be the quotient of $\mathcal{M}_C(2,\OO{C})$ by $\sigma$. Then $Y$ is a holomorphic symplectic variety of dimension 12 whose singular locus is given by $\pi(\operatorname{Fix}(\sigma))$ and is pure of dimension 10. The chain of closed immersions
\begin{equation}\label{stratification-Y}
Y \supseteq \pi(\operatorname{Fix}(\sigma)) \supseteq \pi(\Sigma(\mathcal{M}_C(2,\OO{C}))) \supseteq \pi(\Omega(\mathcal{M}_C(2,\OO{C})))
\end{equation}
defines the stratification of $Y$ by symplectic leaves of Proposition \ref{stratification-symplectic-leaves}.
\end{lem}

\begin{proof} Since $Y$ is a holomorphic symplectic variety by Example \ref{examples-HSV}, the statement follows from Lemma \ref{fixed-locus-of-sigma} and Proposition \ref{stratification-symplectic-leaves}.
\end{proof}

Now let $\mathcal{I}\subseteq \OO{Y}$ be the ideal sheaf of the reduced singular locus, i.e. of  $\pi(\operatorname{Fix}(\sigma))\subseteq Y$ with its reduced subscheme structure. Consider the corresponding blow-up
\begin{equation}\label{blow-up-resolution}
f \colon \tilde{Y}\coloneqq \operatorname{Bl}_\mathcal{I}(Y) \to Y.
\end{equation}

We will show that $f$ is a crepant resolution of $Y$ by proving this holds locally over $Y$. 

\begin{lem}\label{discrepancy-clopen} Let $X$ be a normal algebraic variety equipped with a stratification by locally closed connected smooth subvarieties $X_i \subseteq X$ such that $X$ is normally flat along each stratum. Assume $f\colon \tilde{X}\to X$ is a blow-up of $X$ along a closed union of strata. Then any integral component of the exceptional divisor of $f$ surjects onto the closure of a stratum. 
\end{lem}

\begin{proof} Let $E$ be an integral exceptional divisor and let $X_i$ be a stratum containing the generic point of $f(E)$. We want to show that $f(E)\cap X_i=X_i$.  If this is not the case, then $\dim f(E)\cap X_i < \dim X_i$. The intersection $f(E)\cap X_i$ contains an open dense subset of $f(E)$, and hence $f^{-1}(f(E)\cap X_i)$ contains an open dense subset of $E$. 
By our normal flatness assumption, the restriction $f\colon f^{-1}(X_i)\to X_i$ is flat and hence
\begin{equation}
\dim f^{-1}(f(E)\cap X_i)< \dim f^{-1}(X_i) \leq \dim E,
\end{equation}
which is a contradiction. 
\end{proof}

By Remark \ref{normal_flatness}, we can apply the result to the stratification by symplectic leaves of a symplectic variety given by Proposition \ref{stratification-symplectic-leaves}:

\begin{cor}\label{locus-non-crepant-resolution-clopen} Let $X$ be a holomorphic symplectic variety and $f\colon \tilde{X}\to X$ a blow-up of $X$ along a closed union of symplectic leaves which is an isomorphism over $X^{\operatorname{reg}}$. Then the locus of points of $X$ at which $f$ is not a crepant resolution is a closed union of symplectic leaves. 
\end{cor}

\begin{proof} It follows immediately from the properties of the stratification by symplectic leaves in Proposition \ref{stratification-symplectic-leaves} that the image of the reduced singular locus of $\tilde{X}$ is a closed union of strata. Hence we can restrict to the open subset where $f$ is a resolution of singularities and apply Lemma \ref{discrepancy-clopen} to get that the discrepancy centre is also a closed union of strata. 
\end{proof}

We are finally ready to prove Theorem \ref{main_thm_SL}:

\begin{proof}[Proof of Theorem \ref{main_thm_SL}] 
Note that the reduced singular locus of $Y$ is a closed union of symplectic leaves according to Lemma \ref{singularities-quotient}, so by Corollary \ref{locus-non-crepant-resolution-clopen},  it is enough to show that each connected component of the strata contains a point at which $f$ is a crepant resolution. This is immediate for smooth points of $Y$. 

By a standard trick due to Cartan \cite[\textsection 4]{Cartan}, we can linearise and diagonalise the $\Ztwo$-action around a smooth point of $\operatorname{Fix}(\sigma)$, so that by Lemma \ref{fixed-locus-of-sigma} the moduli space around that point is equivariantly isomorphic to a neighbourhood of $(0,0)$ in $\mathbb{C}^{10}\times \mathbb{C}^2$, where $\Ztwo$ acts on the second factor as $-1$. In other words, each connected component of the open stratum in \eqref{stratification-Y} given by $\pi(\operatorname{Fix}(\sigma))$ is a family of canonical surface singularities, and is hence crepantly resolved by $f$.

Since $\pi(\Sigma(\mathcal{M}_C(2,\OO{C})))$ is irreducible, the corresponding open stratum is connected, while the deepest stratum is given by the points of  $\pi(\Omega(\mathcal{M}_C(2,\OO{C})))$. Since we have a $\Pic^0(C)[2]$-action on $\Omega(\mathcal{M}_C(2,\OO{C}))$ which is transitive and commutes with $\sigma$ (Remark \ref{action-of-2-torsion}), it is enough to prove that $f$ is a crepant resolution in an analytic neighbourhood of any point in $\pi(\Omega(\mathcal{M}_C(2,\OO{C})))$. In fact, such a neighbourhood will also intersect the open stratum given by $\pi(\Sigma(\mathcal{M}_C(2,\OO{C})))$.

By Lemma \ref{tangent-action-sigma-tau} and Remark \ref{covering-involution-mu} the involution $\sigma$ on $(\mathcal{M}_C(2,\OO{C}), (\OO{C}^{\oplus 2}, 0))$ agrees with the covering involution of the map 
\begin{equation}
\mu\colon (\mathcal{M}_C(2,\OO{C}), (\OO{C}^{\oplus 2}, 0)) \to (\mathcal{N}_{\leq 3}(V), 0)
\end{equation}
of Proposition \ref{local-structure-SL}. In particular, $\mu$ descends to an isomorphism of analytic germs of symplectic varieties
\begin{equation}\label{crucial-isomorphism-of-germs}
\mu\colon (Y, \pi(\OO{C}^{\oplus 2}, 0)) \to (\mathcal{N}_{\leq 3}(V), 0).
\end{equation}
Note that the stratification \eqref{stratification-Y} of $Y$ gets identified with the stratification \eqref{2-nilpotent-orbits-closure-defn} of $\mathcal{N}_{\leq 3}(V)$ by open orbits of smaller rank. Since the blow-up of the reduced singular locus of $\mathcal{N}_{\leq 3}(V)$ is a crepant resolution by Theorem \ref{blow-up-nilpotent-cone}, the same holds for $Y$ in a neighbourhood of $\pi(\OO{C}^{\oplus 2}, 0)$. 
\end{proof}

\begin{rmk}[Higher genera] It is worth remarking that the assumption on the genus of $C$ in Theorem \ref{main_thm_SL} is essential. When $g\geq 4$, we still get an isomorphism of analytic germs between the quotient $(Y, \pi(\OO{C}^{\oplus 2}, 0))$ and the symplectic orbit closure $(\mathcal{N}_{\leq 3}(V), 0)$, but this time $\dim V\geq 8$, so that the orbit $\mathcal{N}_{\leq 3}(V)$ is not the full 2-nilpotent cone and hence does not admit a crepant resolution according to Theorem \ref{blow-up-nilpotent-cone}. 
\end{rmk}

We now show how to deduce Theorem \ref{main_thm_GL} from Theorem \ref{main_thm_SL}, by exploiting a suitable equivariant version of Remark \ref{triviality-determinant-morphism-Higgs}. 

\begin{lem}[Singularities of the quotient of $\mathcal{M}_C(2,0)$] Let $\pi: \mathcal{M}_C(2,0) \to X$ be the quotient of $\mathcal{M}_C(2,0)$ by $\sigma\circ \tau$. Then $X$ is a holomorphic symplectic variety of dimension 18 whose singular locus is given by $\pi(\operatorname{Fix}(\sigma\circ \tau))$ and is pure of dimension 16. The chain of closed immersions
\begin{equation}\label{whitney-stratification-GL}
X \supseteq \pi(\operatorname{Fix}(\sigma\circ \tau)) \supseteq \pi(\Sigma(\mathcal{M}_C(2,0))) \supseteq \pi(\Omega(\mathcal{M}_C(2,0)))
\end{equation}
defines the stratification of $X$ by symplectic leaves of Proposition \ref{stratification-symplectic-leaves}.
\end{lem}

\begin{proof} Once again, $X$ is a holomorphic symplectic variety by Example \ref{examples-HSV}. We reduce to the case of $\mathcal{M}_C(2, \OO{C})$ by the same argument used in the proof of Theorem \ref{hyperelliptic-moduli-bundles}. Indeed, by Remark \ref{triviality-determinant-morphism-Higgs}, the moduli space $\mathcal{M}_C(2,0)$ looks \'etale locally like a product $\mathcal{M}_C(2,\OO{C})\times T^\ast \Pic^0(C)$. The Galois cover 
\begin{equation}\label{Galois-cover}
\mathcal{M}_C(2,\OO{C})\times T^\ast \Pic^0(C) \to \mathcal{M}_C(2,0)
\end{equation}
is $\Ztwo$-equivariant with respect to the action on $\mathcal{M}_C(2,\OO{C})$ and $\mathcal{M}_C(2,0)$ defined by $\sigma\circ \tau$.

Moreover, \eqref{Galois-cover} is strongly \'etale -- in the sense of \cite[Appendix D to Ch. 1]{GIT} -- with respect to the action of $\sigma\circ \tau$. It induces an \'etale map between the corresponding fixed loci, so that the fixed locus of $\sigma\circ \tau$ in $\mathcal{M}_C(2,0)$ contains the singular locus and is pure of dimension 16 by Lemma \ref{fixed-locus-of-sigma}. Moreover there is a cartesian diagram
\begin{equation} \label{crucial-diagram}
\begin{tikzcd}
\mathcal{M}_C(2, \OO{C})\times T^\ast \Pic^0(C) \ar[r] \ar[d] &  \mathcal{M}_C(2, 0) \ar[d] \\
Y \times T^\ast \Pic^0(C) \ar[r] & X
\end{tikzcd}
\end{equation}
where the vertical arrows are the quotient morphisms and the bottom arrow is \'etale. In particular, the stratification \eqref{whitney-stratification-GL} can be obtained from the corresponding stratification \eqref{stratification-Y} of $Y$ from Lemma \ref{singularities-quotient}. 
\end{proof}

We can now deduce Theorem \ref{main_thm_GL} from Theorem \ref{main_thm_SL}:

\begin{proof}[Proof of Theorem \ref{main_thm_GL}] Let $\mathcal{J}\subseteq \OO{X}$ be the ideal sheaf of the reduced singular locus, i.e. of  $\pi(\operatorname{Fix}(\sigma\circ \tau))\subseteq X$ with its reduced subscheme structure. Consider the corresponding blow-up
\begin{equation}
g \colon \tilde{X}\coloneqq \operatorname{Bl}_{\mathcal{J}}(X) \to X.
\end{equation}
Note that diagram \eqref{crucial-diagram} shows that $X$ and $Y$ are stably isosingular \cite[Def. 2.6]{Mirko}, hence $g\colon \tilde{X}\to X$ is crepant if and only if $f\colon \tilde{Y}\to Y$ is crepant, so the result follows from Theorem \ref{main_thm_SL}. \end{proof}

Note that the determinant morphism from Remark \ref{triviality-determinant-morphism-Higgs} being invariant with respect to $\sigma\circ \tau$, it descends to the quotient. In other words diagram \eqref{crucial-diagram} can be enhanced to
\begin{equation} 
\begin{tikzcd}
\mathcal{M}_C(2, \OO{C})\times T^\ast \Pic^0(C) \ar[r] \ar[d] &  \mathcal{M}_C(2, 0) \ar[d] \ar[dd, bend left, "{(\det, \operatorname{tr})}"] \\
Y \times T^\ast \Pic^0(C) \ar[r] \ar[d] & X \ar[d] \\
T^\ast \Pic^0(C) \ar[r] & T^\ast \Pic^0(C)
\end{tikzcd}
\end{equation}
so that $X\to T^\ast \Pic^0(C)$ is an isotrivial fibration with fibre $Y$.  

\begin{rmk} If one only considers the hyperelliptic involution $\sigma$ acting on $\mathcal{M}_C(2, 0)$ and $\mathcal{M}_C(2, \OO{C})$ via pull-back, the Galois cover \eqref{Galois-cover} is still $\Ztwo$-equivariant, but the induced action on $T^\ast \Pic^0(C)$ is non-trivial. In particular, the singular locus is no longer contained in the fixed locus, and the latter is pure of dimension $10$, so that the quotient has no crepant resolution.  
\end{rmk}

\begin{rmk}[Modular interpretation] In \cite{Seshadri-desingularisation}, Seshadri constructs resolutions of singularities of some moduli spaces of rank two vector bundles on curves which have a modular interpretation: they are moduli spaces of vector bundles with a specific parabolic structure. It would certainly be interesting to know if the symplectic resolutions we consider in Theorem \ref{main_thm_GL} and \ref{main_thm_SL} also have a modular interpretation.
\end{rmk}

{\pagestyle{plain}
\bibliographystyle{halpha}
\bibliography{refs}}

\end{document}